\documentclass[10pt]{article}
\usepackage{amsmath,amscd}
\usepackage{amssymb,latexsym,amsthm}
\usepackage{color}
\usepackage[spanish,english]{babel}
\usepackage{amssymb}
\usepackage{color}
\usepackage{amsmath,amsthm,amscd}
\usepackage[latin1]{inputenc}
\usepackage{lscape}
\usepackage{fancyhdr}
\usepackage{amsfonts}
\usepackage{pb-diagram}

\numberwithin{equation}{section}
\newtheorem{theorem}{Theorem}[section]

\newtheorem{definition}[theorem]{Definition}
\newtheorem{proposition}[theorem]{Proposition}
\newtheorem{lemma}[theorem]{Lemma}

\theoremstyle{definition}

\newtheorem{example}[theorem]{Example}

\newtheorem{remark}[theorem]{Remark}

\newcommand{\cG}{\mbox{${\cal G}$}}

\newcommand{\cO}{\mbox{${\cal O}$}}

\newcommand{\cU}{\mbox{${\cal U}$}}

\newcommand{\cW}{\mbox{${\cal W}$}}

\title{\textbf{Non-commutative algebraic geometry\\ of semi-graded rings}}
\author{Oswaldo Lezama\\
\texttt{jolezamas@unal.edu.co}
\\Edward Latorre
\\ Seminario de Álgebra Constructiva - SAC$^2$\\ Departamento de Matemáticas\\ Universidad Nacional de
Colombia, Sede Bogot\'a}
\date{}
\begin{document}
\maketitle
\begin{abstract}
\noindent In this paper we introduce the semi-graded rings, which extend graded rings and skew PBW
extensions. For this new type of non-commutative rings we will discuss some basic problems of
non-commutative algebraic geometry. In particular, we will prove some elementary properties of the
generalized Hilbert series, Hilbert polynomial and Gelfand-Kirillov dimension. We will extended the
notion of non-commutative projective scheme to the case of semi-graded rings and we generalize the
Serre-Artin-Zhang-Verevkin theorem. Some examples are included at the end of the paper.
\bigskip


\bigskip

\end{abstract}

\section{Introduction}

Finitely graded algebras and skew $PBW$ extensions cover many important classes of non-commutative
rings and algebras coming from quantum mechanics. For example, quantum polynomials are examples of
graded algebras, and universal enveloping algebras of finite-dimensional Lie algebras are examples
of skew $PBW$ (Poincaré-Birkhoff-Witt) extensions (see \cite{Artin}, \cite{Artin2}, \cite{Goetz},
\cite{Rogalski} and \cite{lezamareyes1}). There exists recent interest in developing the
non-commutative projective algebraic geometry for finitely graded algebras (see \cite{Ginzburg1},
\cite{Kanazawa} and \cite{Rogalski}). However, for non $\mathbb{N}$-graded algebras, in particular,
for many important classes of skew $PBW$ extensions, only few works have been realized (see
\cite{Gaddis}). In this paper we present an introduction to non-commutative algebraic geometry for
non $\mathbb{N}$-graded algebras and rings, defining a new class of rings: the \textit{semi-graded
rings}. As we will see, the semi-graded rings generalize the finitely graded algebras and the skew
PBW extensions. We will discuss the most basic problems on non-commutative algebraic geometry for
semi-graded rings. The problems to be discussed are around the following topics: generalized
Hilbert series and Hilbert polynomial, generalized Gelfand-Kirillov dimension, non-commutative
schemes associated to semi-graded rings, and we will extend the Serre-Artin-Zhang-Verevkin theorem
on projective schemes to semi-graded rings satisfying some natural restrictions.

In the present section we will recall the definition of finitely graded algebras and the
Serre-Artin-Zhang-Verevkin theorem (see \cite{Artin2}, \cite{Rogalski}, \cite{Verevkin},
\cite{Verevkin2}). We will also recall the definition of skew $PBW$ extensions and some important
examples of this type of non-commutative rings of polynomial type. In the second section we
introduce the semi-graded rings and modules and we will prove some elementary properties of them.
The third and fourth sections are dedicated to give a generalization of the Hilbert series, Hilbert
polynomial and the classical Gelfand-Kirillov dimension. These three notions are very important in
any approach to non-commutative algebraic geoemtry. The purpose of the fourth section is to
extended the notion of non-commutative projective scheme to the case of semi-graded rings. In the
last section we present a generalization of the Serre-Artin-Zhang-Verevkin theorem about
non-commutative schemes of finitely graded algebras to the case of semi-graded rings of special
type. This main result can be applied to study the non-commutative algebraic geometry of some
important particular examples of quantum algebras, probably not considered before in the
literature.

\begin{definition}
Let $K$ be a field. It is said that a $K$-algebra $A$ is finitely graded if the following
conditions hold:
\begin{itemize}
\item $A$ is $\mathbb{N}$-graded: $A=\bigoplus_{n\geq 0} A_n$.
\item $A$ is connected, i.e., $A_0=K$.
\item $A$ is finitely generated as $K$-algebra.
\end{itemize}
\end{definition}
The most remarkable examples of finitely graded algebras for which the non-commutative projective
algebraic geometry has been developed are the quantum plane, the Jordan plane, the Sklyanin algebra
and the quantum polynomial ring in several variables (see \cite{Rogalski}). Next we review the most
basic facts of the non-commutative algebraic geometry of finitely graded algebras following the
approach given by M. Artin and J. J. Zhang in \cite{Artin2}, by A. B. Verevkin in \cite{Verevkin},
\cite{Verevkin2}, and by Rogalski in \cite{Rogalski}.

Let $A$ be a finitely graded $K$-algebra and $M=\bigoplus_{n\in \mathbb{Z}} M_n$ be a
$\mathbb{Z}$-graded $A$-module which is finitely generated. Then,
\begin{enumerate}
\item[\rm (i)]For every $n\in \mathbb{Z}$, $\dim_{K}M_n<\infty$.
\item[\rm (ii)]The \textit{Hilbert series} of $M$ is defined by
\begin{center}
$h_M(t):=\sum_{n\in \mathbb{Z}}(\dim_K M_n)t^n$.
\end{center}
In particular,
\begin{center}
$h_A(t):=\sum_{n=0}^\infty(\dim_K A_n)t^n$.
\end{center}
\item[\rm (iii)]It says that $A$ has \textit{Hilbert polynomial} if there exists a polynomial $p_A(t)\in \mathbb{Q}[t]$ such that
for all $n$ sufficiently large, $p_A(n)=\dim_K A_n$. In this case $p_A(t)$ is called the Hilbert
polynomial of $A$. Thus,
\begin{center}
$\dim_K A_n=p_A(n)$, for all $n\ggg 0$.
\end{center}
\item[\rm (iv)]The \textit{Gelfand-Kirillov dimension} of $A$ is defined by
\begin{equation}\label{equ17.2.2}
{\rm GKdim}(A):=\sup_{V}\overline{\lim_{n\to \infty}}\log_n\dim_K V^n,
\end{equation}
where $V$ ranges over all frames of $A$ and $V^n:=\, _K\langle v_1\cdots v_n| v_i\in V\rangle$ (a
\textit{frame} of $A$ is a finite dimensional $K$-subspace of $A$ such that $1\in V$; since $A$ is
a $K$-algebra, then $K\hookrightarrow A$, and hence, $K$ is a frame of $A$ of dimension $1$).
\item[\rm (v)]\begin{equation*}
{\rm GKdim}(A)=\overline{\lim_{n\to \infty}}\log_n(\sum_{i=0}^n \dim_K A_i).
\end{equation*}
\item[\rm (vi)]If $p_A(t)$ exists, then
\begin{center}
${\rm GKdim}(A)=\deg(p_A(t))+1$.
\end{center}
\item[\rm (vii)]A famous Serre's theorem on commutative projective algebraic geometry states that the category of coherent sheaves over the projective $n$-space $\mathbb{P}^n$ is
equivalent to a category of noetherian graded modules over a graded commutative polynomial ring.
The study of this equivalence for non-commutative finitely graded noetherian algebras is known as
the non-commutative version of Serre's theorem and is due to Artin, Zhang and Verevkin. In the next
numerals we give the ingredients needed for the formulation of this theorem (see \cite{Artin2},
\cite{Verevkin}).
\item[\rm (viii)]Suppose that $A$ is left noetherian. Let ${\rm gr}-A$ be the abelian category of
finitely generated $\mathbb{Z}$-graded left $A$-modules. It is defined the abelian category ${\rm
qgr}-A$ in the following way: The objects are the same as the objects in ${\rm gr}-A$, and we let
$\pi:{\rm gr}-A\to {\rm qgr}-A$ be the identity map on objects. The morphisms in ${\rm qgr}-A$ are
defined in the following way:
\begin{center}
$Hom_{\rm {qgr}-A}(\pi(M),\pi(N)):=\underrightarrow{\lim}Hom_{{\rm gr}-A}(M_{\geq n}, N/T(N))$,
\end{center}
where the direct limit is taken over maps of abelian groups
\begin{center}
$Hom_{{\rm gr}-A}(M_{\geq n}, N/T(N))\to Hom_{{\rm gr}-A}(M_{\geq n+1}, N/T(N))$
\end{center}
induced by the inclusion homomorphism $M_{\geq n+1}\to M_{\geq n}$; $T(N)$ is the torsion submodule
of $N$ and an element $x\in N$ is torsion if $A_{\geq n}x=0$ for some $n\geq 0$. The pair $(\rm
{qgr}-A,\pi(A))$ is called the \textit{non-commutative projective scheme} associated to $A$, and
denoted by $\rm {qgr}-A$. Thus, ${\rm qgr}-A$ is a quotient category, ${\rm qgr}-A={\rm gr}-A/{\rm
tor}-A$.
\item[\rm (ix)]\textit{$($Serre's theorem$)$ Let $A$ be a commutative finitely graded $K$-algebra generated in degree $1$.
Then, there exists an equivalence of categories}
\begin{center}
${\rm qgr}-A\simeq {\rm coh}({\rm proj}(A))$.
\end{center}
\textit{In particular},
\begin{center}
${\rm qgr}-K[x_0,\dots,x_n]\simeq {\rm coh}(\mathbb{P}^n)$.
\end{center}
\item[\rm (x)]Suppose that $A$ is left noetherian; let $i\geq 0$; it is said that $A$ satisfies the $\chi_i$ condition if for every
finitely generated $\mathbb{Z}$-graded $A$-module $M$, $\dim_K(\underline{Ext}_A^j(K,M))<\infty$
for any $j\leq i$; the algebra $A$ satisfies the $\chi$ condition if it satisfies the $\chi_i$
condition for all $i\geq 0$.
\item[\rm (xi)]\textit{$($Artin-Zhang-Verevkin theorem$)$ If $A$ is left noetherian and satisfies $\chi_1$, then $\pi(A)$ is left noetherian and there exists an equivalence of categories
\begin{center}
${\rm qgr}-A\simeq {\rm qgr}-\Gamma(\pi(A))_{\geq 0}$,
\end{center}
where $\Gamma(\pi(A))_{\geq 0}:=\bigoplus_{d=0}^\infty Hom_{{\rm qgr}-A}(\pi(A),s^d(\pi(A)))$ and
$s$ is the autoequivalence of ${\rm qg}r-A$ defined by the shifts of degrees.}
\end{enumerate}

Now we recall the definition of skew $PBW$ extension defined firstly in \cite{LezamaGallego}; many
important algebras coming from mathematical physics are particular examples of skew $PBW$
extensions: $\cal{U}(\cal{G})$, where $\cal{G}$ is a finite dimensional Lie algebra, the algebra of
$q$-differential operators, the algebra of shift operators, the additive analogue of the Weyl
algebra, the multiplicative analogue of the Weyl algebra, the quantum algebra
$\mathcal{U}'(so(3,K))$, $3$-dimensional skew polynomial algebras, the dispin algebra, the
Woronowicz algebra, the $q$-Heisenberg algebra, are particular examples of skew $PBW$ extensions
(see \cite{lezamareyes1}).

\begin{definition}\label{gpbwextension}
Let $R$ and $A$ be rings. We say that $A$ is a \textit{skew $PBW$ extension of $R$} $($also called
a $\sigma-PBW$ extension of $R$$)$ if the following conditions hold:
\begin{enumerate}
\item[\rm (i)]$R\subseteq A$.
\item[\rm (ii)]There exist finitely many elements $x_1,\dots ,x_n\in A$ such $A$ is a left $R$-free module with basis
\begin{center}
${\rm Mon}(A):= \{x^{\alpha}=x_1^{\alpha_1}\cdots x_n^{\alpha_n}\mid \alpha=(\alpha_1,\dots
,\alpha_n)\in \mathbb{N}^n\}$, with $\mathbb{N}:=\{0,1,2,\dots\}$.
\end{center}
The set $Mon(A)$ is called the set of standard monomials of $A$.
\item[\rm (iii)]For every $1\leq i\leq n$ and $r\in R-\{0\}$ there exists $c_{i,r}\in R-\{0\}$ such that
\begin{equation}\label{sigmadefinicion1}
x_ir-c_{i,r}x_i\in R.
\end{equation}
\item[\rm (iv)]For every $1\leq i,j\leq n$ there exists $c_{i,j}\in R-\{0\}$ such that
\begin{equation}\label{sigmadefinicion2}
x_jx_i-c_{i,j}x_ix_j\in R+Rx_1+\cdots +Rx_n.
\end{equation}
Under these conditions we will write $A:=\sigma(R)\langle x_1,\dots ,x_n\rangle$.
\end{enumerate}
\end{definition}
Associated to a skew $PBW$ extension $A=\sigma(R)\langle x_1,\dots ,x_n\rangle$, there are $n$
injective endomorphisms $\sigma_1,\dots,\sigma_n$ of $R$ and $\sigma_i$-derivations, as the
following proposition shows.
\begin{proposition}\label{sigmadefinition}
Let $A$ be a skew $PBW$ extension of $R$. Then, for every $1\leq i\leq n$, there exist an injective
ring endomorphism $\sigma_i:R\rightarrow R$ and a $\sigma_i$-derivation $\delta_i:R\rightarrow R$
such that
\begin{center}
$x_ir=\sigma_i(r)x_i+\delta_i(r)$,
\end{center}
for each $r\in R$.
\end{proposition}
\begin{proof}
See \cite{LezamaGallego}, Proposition 3.
\end{proof}

A particular case of skew $PBW$ extension is when all derivations $\delta_i$ are zero. Another
interesting case is when all $\sigma_i$ are bijective and the constants $c_{ij}$ are invertible. We
recall the following definition (cf. \cite{LezamaGallego}).
\begin{definition}\label{sigmapbwderivationtype}
Let $A$ be a skew $PBW$ extension.
\begin{enumerate}
\item[\rm (a)]
$A$ is quasi-commutative if the conditions {\rm(}iii{\rm)} and {\rm(}iv{\rm)} in Definition
\ref{gpbwextension} are replaced by
\begin{enumerate}
\item[\rm (iii')]For every $1\leq i\leq n$ and $r\in R-\{0\}$ there exists $c_{i,r}\in R-\{0\}$ such that
\begin{equation}
x_ir=c_{i,r}x_i.
\end{equation}
\item[\rm (iv')]For every $1\leq i,j\leq n$ there exists $c_{i,j}\in R-\{0\}$ such that
\begin{equation}
x_jx_i=c_{i,j}x_ix_j.
\end{equation}
\end{enumerate}
\item[\rm (b)]$A$ is bijective if $\sigma_i$ is bijective for
every $1\leq i\leq n$ and $c_{i,j}$ is invertible for any $1\leq i<j\leq n$.
\end{enumerate}
\end{definition}
Observe that quasi-commutative skew $PBW$ extensions are $\mathbb{N}$-graded rings but arbitrary
skew $PBW$ extensions are semi-graded rings as we will see below. Actually, the main motivation for
constructing the non-commutative algebraic geometry of semi-graded rings is due to arbitrary skew
$PBW$ extensions.

Many properties of skew $PBW$ extensions have been studied in previous works (see
\cite{Lezama-OreGoldie}, \cite{lezamareyes1}). For example, the global, Krull and Goldie dimensions
of bijective skew \textit{PBW} extensions were estimated in \cite{lezamareyes1}. The next theorem
establishes two classical ring theoretic results for skew $PBW$ extensions.

\begin{theorem}\label{1.3.4}
Let $A$ be a bijective skew $PBW$ extension of a ring $R$.
\begin{enumerate}
\item[\rm (i)]$($Hilbert Basis Theorem$)$ If $R$ is a left $($right$)$ Noetherian ring then $A$ is also left $($right$)$ Noetherian.
\item[\rm (ii)]$($Ore's theorem$)$ If $R$ is a left Ore
domain $R$. Then $A$ is also a left Ore domain.
\end{enumerate}
\end{theorem}
We conclude this introductory section fixing some notation: If not otherwise noted, all modules are
left modules; $B$ will denote a non-commutative ring; $K$ will be a field; $A:=\sigma(R)\langle
x_1,\dots ,x_n\rangle$ will represent a skew $PBW$ extension.

\section{Semi-graded rings and modules}

\noindent In this section we introduce the semi-graded rings and modules, we prove some elementary
properties of them, and we will show that graded rings, finitely graded algebras and skew $PBW$
extensions are particular cases of this new type of non-commutative rings.

\begin{definition}
Let $B$ be a ring. We say that $B$ is semi-graded $(SG)$ if there exists a collection
$\{B_n\}_{n\geq 0}$ of subgroups $B_n$ of the additive group $B^+$ such that the following
conditions hold:
\begin{enumerate}
\item[\rm (i)]$B=\bigoplus_{n\geq 0}B_n$.
\item[\rm (ii)]For every $m,n\geq 0$, $B_mB_n\subseteq B_0\oplus \cdots \oplus B_{m+n}$.
\item[\rm (iii)]$1\in B_0$.
\end{enumerate}
The collection $\{B_n\}_{n\geq 0}$ is called a semi-graduation of $B$ and we say that the elements
of $B_n$ are homogeneous of degree $n$. Let $B$ and $C$ be semi-graded rings and let $f: B\to C$ be
a ring homomorphism, we say that $f$ is homogeneous if $f(B_n)\subseteq C_{n}$ for every $n\geq 0$.
\end{definition}

\begin{definition}
Let $B$ be a $SG$ ring and let $M$ be a $B$-module. We say that $M$ is a $\mathbb{Z}$-semi-graded,
or simply semi-graded, if there exists a collection $\{M_n\}_{n\in \mathbb{Z}}$ of subgroups $M_n$
of the additive group $M^+$ such that the following conditions hold:
\begin{enumerate}
\item[\rm (i)]$M=\bigoplus_{n\in \mathbb{Z}} M_n$.
\item[\rm (ii)]For every $m\geq 0$ and $n\in \mathbb{Z}$, $B_mM_n\subseteq \bigoplus_{k\leq m+n}M_k$.
\end{enumerate}
We say that $M$ is positively semi-graded, also called $\mathbb{N}$-semi-graded, if $M_n=0$ for
every $n<0$. Let $f: M\to N$ be an homomorphism of $B$-modules, where $M$ and $N$ are semi-graded
$B$-modules; we say that $f$ is homogeneous if $f(M_n)\subseteq N_n$ for every $n\in \mathbb{Z}$.
\end{definition}

As for the case of rings, the collection $\{M_n\}_{n\in \mathbb{Z}}$ is called a
\textit{semi-graduation} of $M$ and we say that the elements of $M_n$ are \textit{homogeneous} of
degree $n$.

Let $B$ be a semi-graded ring and let $M$ be a semi-graded $B$-module, let $N$ be a submodule of
$M$, let $N_n:=N\cap M_n$, $n\in \mathbb{Z}$; observe that the sum $\sum_{n}N_n$ is direct. This
induces the following definition.

\begin{definition}
Let $B$ be a $SG$ ring and $M$ be a semi-graded module over $B$. Let $N$ be a submodule of $M$, we
say that $N$ is a semi-graded submodule of $M$ if $N=\bigoplus_{n\in \mathbb{Z}}N_n$.
\end{definition}

Note that if $N$ is semi-graded, then $B_mN_n\subseteq \bigoplus_{k\leq m+n}N_k$, for every $n\in
\mathbb{Z}$ and $m\geq 0$: In fact, let $b\in B_m$ and $z\in N_n$, then $bz\in B_mM_n\subseteq
\bigoplus_{k\leq m+n}M_k$ and $bz=z_1+\cdots+z_l$, with $z_i\in N_{n_i}\subseteq M_{n_i}$, but
since the sum is direct, then $n_i\leq m+n$ for every $1\le i\leq l$.

Finally, we introduce an important class of semi-graded rings that includes finitely graded
algebras and skew $PBW$ extensions.

\begin{definition}\label{definition17.5.4}
Let $B$ be a ring. We say that $B$ is finitely semi-graded $(FSG)$ if $B$ satisfies the following
conditions:
\begin{enumerate}
\item[\rm (i)]$B$ is $SG$.
\item[\rm (ii)]There exists finitely many elements $x_1,\dots,x_n\in B$ such that the
subring generated by $B_0$ and $x_1,\dots,x_n$ coincides with $B$.
\item[\rm (iii)]For every $n\geq 0$, $B_n$ is a free $B_0$-module of finite dimension.
\end{enumerate}
Moreover, if $M$ is a $B$-module, we say that $M$ is finitely semi-graded if $M$ is semi-graded,
finitely generated, and for every $n\in \mathbb{Z}$, $M_n$ is a free $B_0$-module of finite
dimension.
\end{definition}

\begin{remark}
Observe if $B$ is $FSG$, then $B_0B_p=B_p$ for every $p\geq 0$, and if $M$ is finitely semi-graded,
then $B_0M_n=M_n$ for all $n\in \mathbb{Z}$.
\end{remark}

From the definitions above we get the following conclusions.

\begin{proposition}\label{proposition17.5.5}
Let $B=\bigoplus_{n\geq 0}B_n$ be a $SG$ ring and $I$ be a proper two-sided ideal of $B$
semi-graded as left ideal. Then,
\begin{enumerate}
\item[\rm (i)]$B_0$ is a subring of $B$. Moreover, for any $n\geq 0$, $B_0\oplus \cdots \oplus B_{n}$ is a $B_0-B_0$-bimodule, as well as $B$.
\item[\rm (ii)]$B$ has a standard $\mathbb{N}$-filtration given by
\begin{equation}\label{equ17.5.1}
F_n(B):=B_0\oplus \cdots \oplus B_{n}.
\end{equation}
\item[\rm (iii)]The associated graded ring $Gr(B)$ satisfies
\begin{center}
$Gr(B)_n\cong B_n$, for every $n\geq 0$ $($isomorphism of abelian groups$)$.
\end{center}
\item[\rm (iv)]Let $M=\bigoplus_{n\in \mathbb{Z}}M_n$ be a semi-graded $B$-module and $N$ a submodule of $M$. The following conditions are equivalent:
\begin{enumerate}
\item[\rm (a)]$N$ is semi-graded.
\item[\rm (b)]For every $z\in N$, the homogeneous components of $z$ are in $N$.
\item[\rm (c)]$M/N$ is semi-graded with semi-graduation given by
\begin{center}
$(M/N)_n:=(M_n+N)/N$, $n\in \mathbb{Z}$.
\end{center}
\end{enumerate}
\item[\rm (v)]$B/I$ is $SG$.
\item[\rm (vi)]If $B$ is $FSG$ and $I\cap B_n\subseteq IB_n$ for
every $n$, then $B/I$ is $FSG$.
\end{enumerate}
\end{proposition}
\begin{proof}
(i) and (ii) are obvious. For (iii) observe that $Gr(B)_n=F_n(B)/F_{n-1}(B)\cong B_n$ for every
$n\geq 0$ (isomorphism of abelian groups); in addition, note how acts the product: let
$z:=\overline{b_0+\cdots+b_m}\in Gr(B)_m$, $z':=\overline{c_0+\cdots+c_n}\in Gr(B)_n$, then
\begin{center}
$zz'=\overline{b_m}\overline{c_n}=\overline{d_0+\cdots+d_{m+n}}=\overline{d_{m+n}}\in
Gr(B)_{m+n}\cong B_{n+m}$.
\end{center}
(iv) (a)$\Leftrightarrow$(b) is obvious.

(b)$\Rightarrow$(c): Let $\overline{M}_n:=(M/N)_n:=(M_n+N)/N$, $n\in \mathbb{Z}$, then
$\overline{M}:=M/N=\bigoplus_{n\in \mathbb{Z}}\overline{M}_n$. In fact, let $z\in M$, then
$\overline{z}\in \overline{M}$ can be written as
$\overline{z}=\overline{z_1+\cdots+z_l}=\overline{z_1}+\cdots+\overline{z_l}$, with $z_k\in
M_{n_k}$, $1\leq k\leq l$, thus, $\overline{z}\in \sum_{n\in \mathbb{Z}}\overline{M}_n$, and hence,
$\overline{M}=\sum_{n\in \mathbb{Z}}\overline{M}_n$. This sum is direct since if
$\overline{z_1}+\cdots+\overline{z_l}=\overline{0}$, then $z_1+\cdots+z_l\in N$, so by (b) $z_k\in
N$, i.e., $\overline{z_k}=\overline{0}$ for every $1\leq k\leq l$. Now, let $b_m\in B_m$ and
$\overline{z_n}\in \overline{M}_n$, then
$b_m\overline{z_n}=\overline{b_mz_n}=\overline{d_1+\cdots+d_{p}}$, with $d_i\in M_{n_i}$ and
$n_i\leq m+n$, so $\overline{b_m}\overline{z_n}=\overline{d_1}+\cdots+\overline{d_{p}}\in
\bigoplus_{k\leq m+n}\overline{M}_k$. We have proved that $\overline{M}$ is semi-graded.

(c)$\Rightarrow$(b): Let $z=z_1+\cdots+z_l\in N$, with $z_i\in M_{n_i}$, $1\leq i\leq l$, then
$\overline{0}=\overline{z_1}+\cdots+\overline{z_l}\in \overline{M}=\bigoplus_{n\in
\mathbb{Z}}\overline{M}_n$, therefore, $\overline{z_i}=\overline{0}$, and hence $z_i\in N$ for
every $i$.

(v) The proof is similar to (b)$\Rightarrow$(c) in (iv).

(vi) By (v), $\overline{B}$ is $SG$. Let $x_1,\dots,x_n\in B$ such that the subring generated by
$B_0$ and $x_1,\dots,x_n$ coincides with $B$, then it is clear that the subring of $\overline{B}$
generated by $\overline{B}_0$ and $\overline{x_1},\dots,\overline{x_n}$ coincides with
$\overline{B}$. Let $n\geq 0$ and $\{z_1,\dots,z_l\}$ be a basis of the free left $B_0$-module
$B_n$, then $\{\overline{z_1},\dots,\overline{z_l}\}$ is a basis of $\overline{B}_n$: in fact, let
$\overline{z}\in \overline{B}_n$ with $z\in B_n$, then $z=c_1z_1+\cdots+c_lz_l$, with $c_i\in B_0$,
$1\leq i\leq l$, and hence,
$\overline{z}=\overline{c_1}\,\overline{z_1}+\cdots+\overline{c_l}\,\overline{z_l}$, i.e.,
$\overline{B}_n$ is generated by $\overline{z_1},\dots,\overline{z_l}$ over $\overline{B}_0$; now,
if $\overline{c_1}\,\overline{z_1}+\cdots+\overline{c_l}\,\overline{z_l}=\overline{0}$, with
$c_i\in B_0$, then $c_1z_1+\cdots+c_lz_l\in I\cap B_n\subseteq IB_n$ and hence we can write
\begin{center}
$c_1z_1+\cdots+c_lz_l=d_1z_1+\cdots+d_lz_l$, with $d_i\in I$, $1\leq i\leq l$.
\end{center}
From this we get that $c_i=d_i$, so $\overline{c_i}=\overline{0}$ for every $i$.
\end{proof}
Note that the condition imposed to $I$ in (vi) is of type Artin-Rees (see \cite{McConnell}).

\begin{proposition}\label{proposition16.5.7}
{\rm (i)} Any $\mathbb{N}$-graded ring is $SG$.

{\rm (ii)} Let $K$ be a field. Any finitely graded $K$-algebra is a $FSG$ ring.

{\rm (iii)} Any skew $PBW$ extension is a $FSG$ ring.
\end{proposition}
\begin{proof}
(i) and (ii) follow directly from the definitions.

(iii) Let $A=\sigma(R)\langle x_1,\dots,x_n\rangle$ be a skew $PBW$ extension, then
$A=\bigoplus_{k\geq 0} A_k$, where
\begin{center}
$A_k:=_R\langle x^\alpha\in Mon(A)|\deg(x^\alpha)=:\alpha_1+\cdots+\alpha_n=k\rangle$.
\end{center}
Thus, $A_k$ is a free left $R$-module with
\begin{equation}
\dim_R A_k=\binom{n+k-1}{k}=\binom{n+k-1}{n-1}.
\end{equation}
We are assuming that $R$ is an $\mathcal{IBN}$ ring (\textit{Invariant basis number}), and hence,
$A$ also satisfies this condition, see \cite{Gallego4}.
\end{proof}

\begin{remark}
(i) Note that the class of $SG$ rings includes properly the class of $\mathbb{N}$-graded rings: In
fact, the enveloping algebra of finite-dimensional Lie algebra (see Example \ref{example6.14}
below) proves this statement.

(ii) The example in (i) proves also that the class of $FSG$ rings includes properly the class of
finitely graded algebras.

(iii) Finally, the class of $FSG$ rings includes properly the class of skew $PBW$ extensions: For
this consider the Artin-Schelter regular algebra of global dimension $3$ defined by the following
relations:
\begin{center}
$yx=xy+z^2$, $zy=yz+x^2$, $zx=xz+y^2$.
\end{center}
Observe that this algebra is a particular case of a Sklyanin algebra wich in general are defined by
the following relations:
\begin{center}
$ayx+bxy+cz^2=0$, $azy+byz+cx^2=0$, $axz+bzx+cy^2=0$, $a,b,c\in K$.
\end{center}
\end{remark}

\section{Generalized Hilbert series and Hilbert polynomial}

In this section we introduce the notion of generalized Hilbert series and generalized Hilbert
polynomial for semi-graded rings. As in the classical case of finitely graded algebras over fields,
these notions depends on the semi-graduation, in particular, they depend on the ring $B_0$. We will
compute these tools for skew $PBW$ extensions.

\begin{definition}
Let $B=\sum_{n\geq 0}\oplus B_n$ be a $FSG$ ring and $M=\bigoplus_{n\in \mathbb{Z}}M_n$ be a
finitely semi-graded $B$-module. The generalized Hilbert series of $M$ is defined by
\begin{center}
$Gh_M(t):=\sum_{n\in \mathbb{Z}}(\dim_{B_0} M_n)t^n$.
\end{center}
In particular,
\begin{center}
$Gh_B(t):=\sum_{n=0}^\infty(\dim_{B_0} B_n)t^n$.
\end{center}
We say that $B$ has a generalized Hilbert polynomial if there exists a polynomial $Gp_B(t)\in
\mathbb{Q}[t]$ such that
\begin{center}
$\dim_{B_0} B_n=Gp_B(n)$, for all $n\ggg 0$.
\end{center}
In this case $Gp_B(t)$ is called the generalized Hilbert polynomial of $B$.
\end{definition}

\begin{remark}\label{remark17.5.11}
(i) Note that if $K$ is a field and $B$ is a finitely graded $K$-algebra, then the generalized
Hilbert series coincides with the habitual Hilbert series, i.e., $Gh_{B}(t)=h_B(t)$; the same is
true for the generalized Hilbert polynomial.

(ii) Observe that if a semi-graded ring $B$ has another \-se\-mi-\-gra\-dua\-tion
$B=\bigoplus_{n\geq 0} C_n$, then its generalized Hilbert series and its generalized Hilbert
polynomial can change, i.e., the notions of generalized Hilbert series and generalized Hilbert
polynomial depend on the semi-graduation, in particular on $B_0$. For example, consider the
habitual real polynomial ring in two variables $B:=\mathbb{R}[x,y]$, then
$Gh_{B}(t)=\frac{1}{(1-t)^2}$ and $Gp_{B}(t)=t+1$; but if we view this ring as
$B=(\mathbb{R}[x])[y]$ then $C_0=\mathbb{R}[x]$, its generalized Hilbert polynomial series is
$\frac{1}{1-t}$ and its generalized Hilbert polynomial is $1$.
\end{remark}

For skew $PBW$ extensions the generalized Hilbert series and the generalized Hilbert polynomial can
be computed explicitly.

\begin{theorem}\label{17.5.10}
Let $A=\sigma(R)\langle x_1,\dots,x_n\rangle$ be an arbitrary skew $PBW$ extension. Then,
\begin{enumerate}
\item[\rm (i)]
\begin{equation}\label{equ17.2.3}
Gh_A(t)=\frac{1}{(1-t)^n}.
\end{equation}
\item[\rm (ii)]
\begin{equation}\label{equation17.5.4}
Gp_A(t)=\frac{1}{(n-1)!}[t^{n-1}-s_1t^{n-2}+\cdots+(-1)^rs_{r}t^{n-r-1}+\cdots+(n-1)!],
\end{equation}
where $s_1,\dots,s_k,\dots,s_{n-1}$ are the elementary symmetric polynomials in the variables
$1-n,2-n,\dots,(n-1)-n$.
\end{enumerate}
\end{theorem}
\begin{proof}
(i) We have
\begin{equation*}
Gh_A(t)=\sum_{k=0}^\infty(\dim_R A_k)t^k=\sum_{k=0}^\infty \binom{n+k-1}{k}t^k=\frac{1}{(1-t)^n}.
\end{equation*}
(ii) Note that {\tiny
\begin{align*}
\dim_R A_k=\binom{n+k-1}{k}& =\frac{(n+k-1)!}{k!(n-1)!}\\
& =\frac{(k+n-1)(k+n-2)(k+n-3)\cdots (k+n-(n-1))(k+n-n)!}{k!(n-1)!}\\
& =\frac{(k+n-1)(k+n-2)(k+n-3)\cdots (k+n-(n-1))}{(n-1)!}\\
& =\frac{1}{(n-1)!}[k^{n-1}-s_1k^{n-2}\cdots+(-1)^rs_{r}k^{n-r-1}+\cdots+(-1)^{n-1}s_{n-1}k^{n-n}]\\
& =\frac{1}{(n-1)!}[k^{n-1}-s_1k^{n-2}\cdots+(-1)^rs_{r}k^{n-r-1}+\cdots+(n-1)!],
\end{align*}}
where $s_1,\dots,s_k,\dots,s_{n-1}$ are the elementary symmetric polynomials in the variables
$1-n,2-n,\dots,(n-1)-n$. Thus, we found a polynomial $Gp_A(t)\in \mathbb{Q}[t]$ of degree $n-1$
such that
\begin{equation}\label{equ17.5.5}
\dim_R A_k=Gp_A(k)\ \text{for all $k\geq 0$}.
\end{equation}
\end{proof}

From Theorem \ref{17.5.10}, and considering the numeral (ii) in Remark \ref{remark17.5.11}, we can
compute the generalized Hilbert series and the generalized Hilbert polynomial for all examples of
skew $PBW$ extensions described in \cite{lezamareyes1}. In addition, for the skew quantum
polynomials, we can interpreted some of them as quasi-commutative bijective skew $PBW$ extensions
of the $r$-multiparameter quantum torus. Thus, we have the following tables:

\newpage

\begin{center}
\begin{table}[htb]
\centering \tiny{
\begin{tabular}{|l|l|l|}\hline 
\textbf{Ring} & \textbf{$Gh_A(t)$} & \textbf{$Gp_A(t)$}
\\ \hline \hline Habitual polynomial ring $R[x_1,\dotsc,x_n]$ & $\frac{1}{(1-t)^n}$ &  $\frac{1}{(n-1)!}[t^{n-1}+\cdots+1]$\\
\cline{1-3} Ore extension of bijective type $R[x_1;\sigma_1 ,\delta_1]\cdots [x_n;\sigma_n
,\delta_n]$
 & $\frac{1}{(1-t)^n}$ &  $\frac{1}{(n-1)!}[t^{n-1}+\cdots+1]$\\
\cline{1-3} Weyl algebra $A_n(K)$ & $\frac{1}{(1-t)^n}$ &  $\frac{1}{(n-1)!}[t^{n-1}+\cdots+1]$\\
\cline{1-3}
Extended Weyl algebra $B_n(K)$ & $\frac{1}{(1-t)^n}$ & $\frac{1}{(n-1)!}[t^{n-1}+\cdots+1]$\\
\cline{1-3} Enveloping algebra of a Lie algebra $\mathcal{G}$ of dimension $n$,
$\cU(\mathcal{G})$ & $\frac{1}{(1-t)^n}$ &  $\frac{1}{(n-1)!}[t^{n-1}+\cdots+1]$\\
\cline{1-3} Tensor product $R\otimes_K \cU(\cG)$ & $\frac{1}{(1-t)^n}$ &  $\frac{1}{(n-1)!}[t^{n-1}+\cdots+1]$\\
\cline{1-3} Crossed product $R*\cU(\cG)$ & $\frac{1}{(1-t)^n}$ &  $\frac{1}{(n-1)!}[t^{n-1}+\cdots+1]$\\
\cline{1-3} Algebra of q-differential operators $D_{q,h}[x,y]$
& $\frac{1}{1-t}$ & $1$ \\
\cline{1-3} Algebra of shift operators $S_h$ & $\frac{1}{1-t}$ & $1$ \\
\cline{1-3}
Mixed algebra $D_h$ & $\frac{1}{(1-t)^2}$ &  $t+1$\\
\cline{1-3} Discrete linear systems
$K[t_1,\dotsc,t_n][x_1,\sigma_1]\dotsb[x_n;\sigma_n]$ & $\frac{1}{(1-t)^n}$ & $\frac{1}{(n-1)!}[t^{n-1}+\cdots+1]$\\
\cline{1-3} Linear partial shift operators
$K[t_1,\dotsc,t_n][E_1,\dotsc,E_n]$ & $\frac{1}{(1-t)^n}$ &  $\frac{1}{(n-1)!}[t^{n-1}+\cdots+1]$\\
\cline{1-3} Linear partial shift operators
$K(t_1,\dotsc,t_n)[E_1,\dotsc,E_n]$ & $\frac{1}{(1-t)^n}$ &  $\frac{1}{(n-1)!}[t^{n-1}+\cdots+1]$\\
\cline{1-3} L. P. Differential operators
$K[t_1,\dotsc,t_n][\partial_1,\dotsc,\partial_n]$ & $\frac{1}{(1-t)^n}$ &  $\frac{1}{(n-1)!}[t^{n-1}+\cdots+1]$\\
\cline{1-3} L. P. Differential operators
$K(t_1,\dotsc,t_n)[\partial_1,\dotsc,\partial_n]$ & $\frac{1}{(1-t)^n}$ &  $\frac{1}{(n-1)!}[t^{n-1}+\cdots+1]$\\
\cline{1-3} L. P. Difference operators
$K[t_1,\dotsc,t_n][\Delta_1,\dotsc,\Delta_n]$ & $\frac{1}{(1-t)^n}$ &  $\frac{1}{(n-1)!}[t^{n-1}+\cdots+1]$\\
\cline{1-3} L. P. Difference operators
$K(t_1,\dotsc,t_n)[\Delta_1,\dotsc,\Delta_n]$ & $\frac{1}{(1-t)^n}$ &  $\frac{1}{(n-1)!}[t^{n-1}+\cdots+1]$\\
\cline{1-3} L. P. $q$-dilation operators
$K[t_1,\dotsc,t_n][H_1^{(q)},\dotsc,H_m^{(q)}]$ & $\frac{1}{(1-t)^m}$ &  $\frac{1}{(m-1)!}[t^{m-1}+\cdots+1]$\\
\cline{1-3} L. P. $q$-dilation operators
$K(t_1,\dotsc,t_n)[H_1^{(q)},\dotsc,H_m^{(q)}]$ & $\frac{1}{(1-t)^m}$ &  $\frac{1}{(m-1)!}[t^{m-1}+\cdots+1]$\\
\cline{1-3} L. P. $q$-differential operators
$K[t_1,\dotsc,t_n][D_1^{(q)},\dotsc,D_m^{(q)}]$ & $\frac{1}{(1-t)^m}$ &  $\frac{1}{(m-1)!}[t^{m-1}+\cdots+1]$\\
\cline{1-3} L. P. $q$-differential operators
$K(t_1,\dotsc,t_n)[D_1^{(q)},\dotsc,D_m^{(q)}]$ & $\frac{1}{(1-t)^m}$ &  $\frac{1}{(m-1)!}[t^{m-1}+\cdots+1]$\\
\cline{1-3} Diffusion algebras & $\frac{1}{(1-t)^n}$ &  $\frac{1}{(n-1)!}[t^{n-1}+\cdots+1]$\\
\cline{1-3}
Additive analogue of the Weyl algebra $A_n(q_1,\dotsc,q_n)$ & $\frac{1}{(1-t)^n}$ & $\frac{1}{(n-1)!}[t^{n-1}+\cdots+1]$\\
\cline{1-3} Multiplicative analogue of the Weyl algebra
$\cO_n(\lambda_{ji})$ & $\frac{1}{(1-t)^{n-1}}$ &  $\frac{1}{(n-2)!}[t^{n-2}+\cdots+1]$\\
\cline{1-3} Quantum algebra $\cU'(\mathfrak{so}(3,K))$ & $\frac{1}{(1-t)^3}$ &  $\frac{1}{2}[t^{2}+3t+1]$\\
\cline{1-3} 3-dimensional skew polynomial algebras & $\frac{1}{(1-t)^3}$ &  $\frac{1}{2}[t^{2}+3t+1]$\\
\cline{1-3} Dispin algebra $\cU(osp(1,2))$ & $\frac{1}{(1-t)^3}$ &  $\frac{1}{2}[t^{2}+3t+1]$\\
\cline{1-3} Woronowicz algebra $\cW_{\nu}(\mathfrak{sl}(2,K))$ & $\frac{1}{(1-t)^3}$ &  $\frac{1}{2}[t^{2}+3t+1]$\\
\cline{1-3}
Complex algebra $V_q(\mathfrak{sl}_3(\mathbb{C}))$ & $\frac{1}{(1-t)^6}$ &  $\frac{1}{120}[t^{5}+15t^4+85t^3+217t^2+274t+120]$\\
\cline{1-3} Algebra \textbf{U} & $\frac{1}{(1-t)^{2n}}$ &  $\frac{1}{(2n-1)!}[t^{2n-1}+\cdots+1]$\\
\cline{1-3} Manin algebra $\cO_q(M_2(K))$ & $\frac{1}{(1-t)^3}$ &  $\frac{1}{2}[t^{2}+3t+1]$\\
\cline{1-3} Coordinate algebra of the quantum group
$SL_q(2)$ & $\frac{1}{(1-t)^3}$ &  $\frac{1}{2}[t^{2}+3t+1]$\\
\cline{1-3} $q$-Heisenberg algebra \textbf{H}$_n(q)$ & $\frac{1}{(1-t)^{2n}}$ &  $\frac{1}{(2n-1)!}[t^{2n-1}+\cdots+1]$\\
\cline{1-3} Quantum enveloping algebra of
$\mathfrak{sl}(2,K)$, $\cU_q(\mathfrak{sl}(2,K))$ & $\frac{1}{(1-t)^2}$ &  $t+1$\\
\cline{1-3} Hayashi algebra $W_q(J)$ & $\frac{1}{(1-t)^{2n}}$ & $\frac{1}{(2n-1)!}[t^{2n-1}+\cdots+1]$\\
\cline{1-3} Differential operators on a quantum space $S_{\textbf{q}}$,
$D_{\textbf{q}}(S_{\textbf{q}})$ & $\frac{1}{(1-t)^n}$ &  $\frac{1}{(n-1)!}[t^{n-1}+\cdots+1]$\\
\cline{1-3}
Witten's Deformation of $\cU(\mathfrak{sl}(2,K)$ & $\frac{1}{1-t}$ & 1\\
\cline{1-3} Quantum Weyl algebra of Maltsiniotis
$A_n^{\textbf{q},\lambda}$ & $\frac{1}{(1-t)^2}$ & $t+1$\\
\cline{1-3} Quantum Weyl algebra $A_n(q,p_{i,j})$ & $\frac{1}{(1-t)^2}$ & $t+1$\\
\cline{1-3} Multiparameter Weyl algebra $A_n^{Q,\Gamma}(K)$ & $\frac{1}{(1-t)^2}$ & $t+1$\\
\cline{1-3} Quantum
symplectic space $\cO_q(\mathfrak{sp}(K^{2n}))$ & $\frac{1}{(1-t)^2}$ &  $t+1$\\
\cline{1-3} Quadratic algebras in 3 variables & $\frac{1}{1-t}$ &  1\\
\cline{1-3}
\end{tabular}}
\caption{Hilbert series and Hilbert polynomial for some examples of bijective skew $PBW$
extensions.}
\end{table}
\end{center}

\newpage

\begin{center}
\begin{table}[htb]
\centering \tiny{
\begin{tabular}{|l|l|l|}\hline
\textbf{Ring} & \textbf{$Gh_A(t)$} & \textbf{$Gp_A(t)$}
\\ \hline
\hline
\cline{1-3} $n$-Multiparametric skew quantum space $R_{\textbf{q},\sigma}[x_1,\dotsc,x_n]$ & $\frac{1}{(1-t)^n}$ &  $\frac{1}{(n-1)!}[t^{n-1}+\cdots+1]$\\
\cline{1-3} $n$-Multiparametric quantum space $R_{\textbf{q}}[x_1,\dotsc,x_n]$ & $\frac{1}{(1-t)^n}$ &  $\frac{1}{(n-1)!}[t^{n-1}+\cdots+1]$\\
\cline{1-3} $n$-Multiparametric skew quantum space $K_{\textbf{q},\sigma}[x_1,\dotsc,x_n]$ & $\frac{1}{(1-t)^n}$ &  $\frac{1}{(n-1)!}[t^{n-1}+\cdots+1]$\\
\cline{1-3} $n$-Multiparametric quantum space $K_{\textbf{q}}[x_1,\dotsc,x_n]$ & $\frac{1}{(1-t)^n}$ &  $\frac{1}{(n-1)!}[t^{n-1}+\cdots+1]$\\
\cline{1-3} Ring of skew quantum polynomials
$R_{\textbf{q},\sigma}[x_1^{\pm 1},\dotsc,x_r^{\pm 1},x_{r+1},\dotsc,x_n]$ & $\frac{1}{(1-t)^{n-r}}$ &  $\frac{1}{(n-r-1)!}[t^{n-r-1}+\cdots+1]$\\
\cline{1-3} Ring of quantum polynomials $R_{\textbf{q}}[x_1^{\pm 1},\dotsc,x_r^{\pm
1},x_{r+1},\dotsc,x_n]$ & $\frac{1}{(1-t)^{n-r}}$ &  $\frac{1}{(n-r-1)!}[t^{n-r-1}+\cdots+1]$\\
\cline{1-3} Algebra of skew quantum polynomials $K_{\textbf{q},\sigma}[x_1^{\pm
1},\dotsc,x_r^{\pm 1},x_{r+1},\dotsc,x_n]$ & $\frac{1}{(1-t)^{n-r}}$ &  $\frac{1}{(n-r-1)!}[t^{n-r-1}+\cdots+1]$\\
\cline{1-3} Algebra of quantum polynomials
$\cO_{\textbf{q}}=K_{\textbf{q}}[x_1^{\pm 1},\dotsc,x_r^{\pm 1},x_{r+1},\dotsc,x_n]$ & $\frac{1}{(1-t)^{n-r}}$ &  $\frac{1}{(n-r-1)!}[t^{n-r-1}+\cdots+1]$\\
\cline{1-3}
\end{tabular}}
\caption{Hilbert series and Hilbert polynomials of some skew quantum polynomials.}
\end{table}
\end{center}

\section{Generalized Gelfand-Kirillov dimension}

With respect to the Gelfand-Kirillov dimension, the classical definition over fields is not good
since, in general, for a ring $R$ a finitely generated $R$-module is not free. Whence, we have to
replace the classical dimension of free modules with other invariant. Next we will show that for
our purposes the Goldie dimension works properly, assuming that $R$ is a left noetherian domain. A
similar problem was considered in \cite{Beaulieu} for algebras over commutative noetherian domains
replacing the vector space dimension with the reduced rank.

The following two remarks induce our definition.

\begin{enumerate}

\item[\rm (i)]If $R$ is a left noetherian domain, then $R$ is a left Ore domain and hence ${\rm udim}(_RR)=1$.
From this we get the following conclusion: let $V$ be a free $R$-module of finite dimension, i.e.,
$\dim_R V=k$, then ${\rm udim}(V)=k$: in fact, $V\cong R^k$, and from this we obtain ${\rm
udim}(V)={\rm udim}(_RR\oplus \cdots\oplus _RR)={\rm udim}(_RR)+\cdots +{\rm udim}(_RR)=k$ (see
\cite{McConnell}).

\item[\rm (ii)]Let $R$ be a left noetherian domain and $A=\sigma(R)\langle x_1,\dots,x_n\rangle$ be a skew $PBW$ extension of $R$,
then (\ref{equ17.2.3}) takes the following form:
\begin{equation*}
Gh_A(t)=\sum_{k=0}^\infty({\rm udim}A_k)t^k=\sum_{k=0}^\infty(\dim_RA_k)t^k=\sum_{k=0}^\infty
\binom{n+k-1}{k}t^k=\frac{1}{(1-t)^n}.
\end{equation*}
\end{enumerate}

\begin{definition}
Let $B$ be a $FSG$ ring such that $B_0$ is a left noetherian domain. The generalized
Gelfand-Kirillov dimension of $B$ is defined by
\begin{center}
${\rm GGKdim}(B):=\sup_{V}\overline{\lim_{k\to \infty}}\log_k{\rm udim}V^k$,
\end{center}
where $V$ ranges over all frames of $B$ and $V^k:=\, _{B_0}\langle v_1\cdots v_k| v_i\in V\rangle$
$($a frame of $B$ is a finite dimensional $B_0$-free submodule of $B$ such that $1\in V$$)$.
\end{definition}

\begin{remark}
(i) Note that $B$ has at least one frame: $B_0$ is a frame of dimension $1$. We say that $V$ is a
\textit{generating frame} of $B$ if the subring of $B$ generating by $V$ and $B_0$ is $B$. For
example, if $R$ is a left noetherian domain and $A=\sigma(R)\langle x_1,\dots,x_n\rangle$ is a skew
$PBW$ extension of $R$, then $V:=_{R}\langle 1,x_1,\dots,x_n\rangle$ is a generating frame of $A$.

(ii) In a similar way as was observed in Remark \ref{remark17.5.11}, the notion of generalized
Gelfand-Kirillov dimension of a finitely semi-graded ring $B$ depends on the semi-graduation, in
particular, depends on $B_0$. Note that this type of consideration was made in \cite{Beaulieu} for
an alternative notion of the Gelfand-Kirillov dimension using the reduced rank.

(iii) If $B$ is a finitely graded $K$-algebra, then the classical Gelfand-Kirillov dimension of $B$
coincides with the just defined above notion, i.e., ${\rm GGKdim}(B)={\rm GKdim}(B)$.
\end{remark}

\begin{proposition}
Let $B$ be a $FSG$ ring such that $B_0$ is a left noetherian domain. Let $V$ be a generating frame
of $B$, then
\begin{equation}\label{equ17.2.4}
{\rm GGKdim}(B)=\overline{\lim_{k\to \infty}}\log_k({\rm udim} V^k).
\end{equation}
Moreover, this equality does not depend on the generating frame $V$.
\end{proposition}
\begin{proof}
It is clear that $\overline{\lim_{k\to \infty}}\log_k({\rm udim}V^k)\leq {\rm GGKdim}(B)$. Let $W$
be any frame of $B$; since $\dim_{B_0} W<\infty$, then there exists $m$ such that $W\subseteq V^m$,
and hence for every $k$ we have $W^k\subseteq V^{km}$, but observe that $V^{km}$ is a finitely
generated left $B_0$-module, and since $B_0$ is left noetherian, then $V^{km}$ is a left noetherian
$B_0$-module, so ${\rm udim}V^{km}<\infty$. From this, ${\rm udim}W^k\leq {\rm udim}V^{km}$.
Therefore, $\log_k ({\rm udim}W^k)\leq \log_k({\rm udim}V^{km})=(1+\log_k m)\log_{km}({\rm
udim}V^{km})$. Since $\overline{\lim_{k\to \infty}}(1+\log_k m)=1$, we get that
$\overline{\lim_{k\to \infty}}\log_k({\rm udim}W^k)\leq \overline{\lim_{k\to \infty}}\log_{km}({\rm
udim}V^{km})$. But observe that $\overline{\lim_{k\to \infty}}\log_{km}({\rm udim}V^{km}) \leq
\overline{\lim_{k\to \infty}}\log_{k}({\rm udim}V^{k})$, whence
\[
{\rm GGKdim}(B)=\sup_{W}\overline{\lim_{k\to \infty}}\log_k({\rm udim}W^k)\leq \overline{\lim_{k\to
\infty}}\log_k({\rm udim}V^k).
\]
The proof of the second statement is completely similar.
\end{proof}

Next we present the main result of the present subsection.

\begin{theorem}
Let $R$ be a left noetherian domain and $A=\sigma(R)\langle x_1,\dots,x_n\rangle$ be a skew $PBW$
extension of $R$. Then,
\begin{equation*}
{\rm GGKdim}(A)=\overline{\lim_{k\to \infty}}\log_k(\sum_{i=0}^k \dim_R A_i)=1+\deg(Gp_A(t))=n.
\end{equation*}
\end{theorem}
\begin{proof}
According to (\ref{equ17.2.4}), ${\rm GGKdim}(A)=\overline{\lim_{k\to \infty}}\log_k({\rm udim}
V^k)$, with
\begin{center}
$V:=_{R}\langle 1,x_1,\dots,x_n\rangle =A_0\oplus A_1$;
\end{center}
note that $V^k\subseteq A_0\oplus A_1\oplus \cdots\oplus A_k$, from this and using
(\ref{equ17.5.5}) we get
\begin{center}
${\rm GGKdim}(A)\leq \overline{\lim_{k\to \infty}}\log_k({\rm udim} (\sum_{i=0}^k\oplus
A_i))=\overline{\lim_{k\to \infty}}\log_k(\sum_{i=0}^k{\rm udim} A_i)=\overline{\lim_{k\to
\infty}}\log_k(\sum_{i=0}^k \dim_R A_i)=\overline{\lim_{k\to \infty}}\log_k(\sum_{i=0}^k Gp_A(i))=
\overline{\lim_{k\to \infty}}\log_k(Gp_A(0)+Gp_A(1)+\cdots+Gp_A(k))$,
\end{center}
but according to (\ref{equation17.5.4}), every coefficient in $Gp_A(t)$ is positive, so $Gp_A(i)$
is positive for every $0\leq i\leq k$, moreover, $Gp_A(i)\leq Gp_A(k)$, so
$Gp_A(0)+Gp_A(1)+\cdots+Gp_A(k)\leq (k+1)Gp_A(k)$ and hence
\begin{center}
${\rm GGKdim}(A)\leq\overline{\lim_{k\to \infty}}\log_k((k+1)Gp_A(k))=\overline{\lim_{k\to
\infty}}\log_k(k+1)+\overline{\lim_{k\to \infty}}\log_k(Gp_A(k))= 1+\overline{\lim_{k\to
\infty}}\log_k(Gp_A(k))$.
\end{center}
Observe that every summand of $Gp_A(k)$ in the bracket of (\ref{equation17.5.4}) is $\leq k^{n-1}$
for $k$ enough large, so $Gp_A(k)\leq \frac{n}{(n-1)!}k^{n-1}$ for $k\ggg 0$ and this implies that
\begin{center}
${\rm GGKdim}(A)\leq 1+\overline{\lim_{k\to \infty}}\log_k\frac{n}{(n-1)!}+\overline{\lim_{k\to
\infty}}\log_kk^{n-1}=1+0+n-1=1+\deg(Gp_A(t))=n$.
\end{center}
Now we have to prove that ${\rm GGKdim}(A)\geq n$. Note that $W:=V^{k^{n-1}}$ is a frame of $A$ and
${\rm udim}W^{k}= {\rm udim} V^{k^n}\geq k^n$, therefore, $\log_k({\rm udim} V^{k^n})\geq
\log_kk^n=n$, and hence
\begin{center}
${\rm GGKdim}(A)\geq \overline{\lim_{k\to \infty}}\log_k({\rm udim} W^{k})=\overline{\lim_{k\to
\infty}}\log_k({\rm udim} V^{k^n})\geq \overline{\lim_{k\to \infty}}\, n=n$.
\end{center}
\end{proof}

\section{Non-commutative schemes associated to $SG$ rings}

The purpose of this section is to extended the notion of non-commutative projective scheme to the
case of semi-graded rings.  We will assume that the ring $B$ satisfies the following conditions:

(C1) $B$ is left noetherian $SG$.

(C2) $B_0$ is left noetherian.

(C3) For every $n$, $B_n$ is a finitely generated left $B_0$-module.

(C4) $B_0\subset Z(B)$.

\begin{remark}
(i)  From (C4) we have that $B_0$ is a commutative noetherian ring.

(ii) All important examples of skew $PBW$ extensions satisfy (C1) and (C2). Indeed, let
$A=\sigma(R)\langle x_1,\dots,x_n\rangle$ be a bijective skew $PBW$ extension of $R$, assuming that
$R$ is left noetherian, then $A$ is also left noetherian (Theorem \ref{1.3.4}); in addition, by
Proposition \ref{proposition16.5.7}, $A$ also satisfies (C3).

(iii) With respect to condition (C4), it is satisfied for finitely graded $K$-algebras since in
such case $B_0=K$. On the other hand, let $A=\sigma(R)\langle x_1,\dots,x_n\rangle$ be a skew $PBW$
extension of $R$, then in general $R=A_0\nsubseteq Z(A)$, unless $A$ be a $K$-algebra, with $A_0=K$
a commutative ring.

(iv) It is important to remark that some results below can be proved without assuming all
conditions (C1)-(C4). For example, in Definition \ref{definition16.5.34} we only need (C1).
\end{remark}

\begin{proposition}
Let ${\rm sgr}-B$ be the collection of all finitely generated semi-graded $B$-modules, then ${\rm
sgr}-B$ is an abelian category where the morphisms are the homogeneous $B$-homomorphisms.
\end{proposition}
\begin{proof}
It is clear that ${\rm sgr}-B$ is a category. ${\rm sgr}-B$ has kernels and co-kernels: Let $M,M'$
be objects of ${\rm sgr}-B$ and let $f:M\to M'$ be an homogeneous $B$-homomorphism; let
$L:=\ker(f)$, since $B$ is left noetherian and $M$ is finitely generated, then $L$ is a finitely
generated semi-graded $B$-module; let $M'/Im(f)$ be the co-kernel of $f$, note that $Im(f)$ is
semi-graded, so $M'/Im(f)$ is a semi-graded finitely generated $B$-module.

${\rm sgr}-B$ is normal and co-normal: Let $f:M\to M'$ be a monomorphism in ${\rm sgr}-B$, then $f$
is the kernel of the canonical homomorphism $j:M'\to M'/Im(f)$. Now, let $f:M\to M'$ be an
epimorphism in ${\rm sgr}-B$, then $f$ is the co-kernel of the inclusion $\iota:\ker(f)\to M'$.

${\rm sgr}-B$ is additive: the trivial module $0$ is an object of ${\rm sgr}-B$; if $\{M_i\}$ is a
finite family of objects of ${\rm sgr}-B$, then its co-product $\bigoplus M_i$ in the category of
left $B$-modules is an object of ${\rm sgr}-B$, with semi-graduation given by
\begin{center}
$(\bigoplus M_i)_p:=\bigoplus(M_i)_p$, $p\in \mathbb{Z}$.
\end{center}
Thus, ${\rm sgr}-B$ has finite co-products. Finally, for any objects $M,M'$ of ${\rm sgr}-B$,
$Mor(M,M')$ is an abelian group and the composition of morphisms is bilinear with respect the
operations in these groups.
\end{proof}

\begin{definition}\label{definition16.5.30}
Let $M$ be an object of ${\rm sgr}-B$.
\begin{enumerate}
\item[\rm (i)]For  $s\geq 0$, $B_{\geq s}$ is the
least two-sided ideal of $B$ that satisfies the following conditions:
\begin{enumerate}
\item[\rm (a)]$B_{\geq s}$ contains $\bigoplus_{p\geq s}B_p$.
\item[\rm (b)]$B_{\geq s}$ is semi-graded as left ideal of
$B$.
\item[\rm (c)]$B_{\geq s}$ is a direct sumand of $B$.
\end{enumerate}
\item[\rm (ii)]An element $x\in M$ is torsion if
there exist $s,n\geq 0$ such that $B_{\geq s}^{\ n}x=0$; the set of torsion elements of $M$ is
denoted by $T(M)$; $M$ is \textit{torsion} if $T(M)=M$ and \textit{torsion-free} if $T(M)=0$.
\item[\rm (iii)]For $s,n\geq 0$, $M_{s,n}$ will denote the
least semi-graded submodule of $M$ containing $B_{\geq s}^nM$.
\end{enumerate}
\end{definition}

\begin{remark}
(i) Observe that if $B$ is $\mathbb{N}$-graded, then $B_{\geq s}=\bigoplus_{p\geq s}B_p$.

(ii) Note that $T(M)$ is a submodule of $M$: In fact, let $x,y\in T(M)$, then there exist
$r,s,n,m\geq 0$ such that $B_{\geq s}^{\ n}x=0$ and $B_{\geq r}^{\ m}y=0$; observe that $B_{\geq
r+s}\subseteq B_{\geq s},B_{\geq r}$, so $B_{\geq r+s}^{\ n+m}x\subseteq B_{\geq s}^{\ n}x=0$ and
$B_{\geq r+s}^{\ n+m}y\subseteq B_{\geq r}^{\ m}y=0$, whence $B_{\geq r+s}^{\ n+m}(x+y)=0$, i.e.,
$x+y\in T(M)$; if $b\in B$, then $B_{\geq s}^{\ n}b\subseteq B_{\geq s}^{\ n}$, so $B_{\geq s}^{\
n}bx\subseteq B_{\geq s}^{\ n}x=0$, i.e., $bx\in T(M)$.

(iii) Since $M$ is noetherian, $M_{s,n}$ is finitely generated, i.e., $M_{s,n}$ is an object of
${\rm sgr}-B$. Moreover, $M/M_{s,n}$ is torsion because $B_{\geq s}^nM\subseteq M_{s,n}$. In
addition, note that $M_{s,n}$ is a direct summand of $M$.

(iv) If we assume that $B$ is a  domain, and hence, a left Ore domain, an alternative notion of
torsion can be defined as in the classical case of commutative domains: An element $x\in M$ is
torsion if there exists $b\neq 0$ in $B$ such that $bx=0$; the set $t(M)$ of torsion elements of
$M$ is in this case also a submodule of $M$. In addition, note that $T(M)\subseteq t(M)$: Since
$B_{\geq s}\neq 0$, let $b\neq 0$ in $B_{\geq s}$, then $b^nx=0$ and $b^n\neq 0$, i.e., $x\in
t(M)$.

(v) It is clear that the collection $\mathcal{T}$ of modules $M$ in ${\rm sgr}-B$ such that
$t(M)=M$ conforms a full subcategory of ${\rm sgr}-B$. Moreover, let $0\rightarrow M' \rightarrow
M\rightarrow M''\rightarrow 0$ be a short exact sequence in ${\rm sgr}-B$; it is obvious that
$t(M)=M$  if and only if $t(M')=M'$ and $t(M'')=M''$, i.e., the collection $\mathcal{T}$ is a Serre
subcategory of ${\rm sgr}-B$. The next lemma shows that this property is satisfied also by the
torsion modules introduced in Definition \ref{definition16.5.30}.
\end{remark}

\begin{theorem}\label{theorem16.6.5}
The collection ${\rm stor}-B$ of torsion modules forms a Serre subcategory of ${\rm sgr}-B$, and
the quotient category
\begin{center}
${\rm qsgr}-B:={\rm sgr}-B/{\rm stor}-B$
\end{center}
is abelian.
\end{theorem}
\begin{proof}
It is obvious that  ${\rm stor}-B$ is a full subcategory of ${\rm sgr}-B$. Let $0\rightarrow M'
\xrightarrow{\iota} M\xrightarrow{j} M''\rightarrow 0$ be a short exact sequence in ${\rm sgr}-B$.

Suppose that $M$ is in ${\rm stor}-B$ and let $x'\in M'$, then $\iota(x')\in M$ and there exist
$s,n\geq 0$ such that $\iota(B_{\geq s}^{\ n}x')=B_{\geq s}^{\ n}\iota(x')=0$, but since $\iota$ is
injective, then $B_{\geq s}^{\ n}x'=0$. This means that $x'\in T(M')$, so $T(M')=M'$, i.e., $M'$ is
in ${\rm stor}-B$. Now let $x''\in M''$, then there exists $x\in M$ such that $j(x)=x''$; there
exist $r,m\geq 0$ such that $B_{\geq r}^{\ m}x=0$, whence $B_{\geq s}^{\ n}x''=0$, this implies
that $x''\in T(M'')$. Thus, $T(M'')=M''$, i.e., $M''$ is in ${\rm stor}-B$.

Conversely, suppose that $M'$ and $M''$ are in ${\rm stor}-B$; let $x\in M$, then there exist
$s,n\geq 0$ such that $B_{\geq s}^{\ n}j(x)=0$, i.e., $j(B_{\geq s}^{\ n}x)=0$. Therefore, $B_{\geq
s}^{\ n}x\subseteq \ker(j)=Im(\iota)$, but since $M'$ is torsion, then $Im(\iota)$ is also a
torsion module. Because $B$ is left noetherian, there exist $a_1,\dots,a_l\in B_{\geq s}^{\ n}$
such that $B_{\geq s}^{\ n}=Ba_1+\cdots+Ba_l$; there exist $r_i,m_i\geq 0$, $1\leq i\leq l$, such
that $B_{\geq r_i}^{\ m_i}a_ix=0$. Without lost of generality we can assume that $r_1\geq r_i$ for
every i, so $B_{\geq r_1}\subseteq B_{\geq r_i}$ and hence $B_{\geq r_1}^{\ m_1}\subseteq B_{\geq
r_1}^{\ m_1}, B_{\geq r_1}^{\ m_2}\subseteq B_{\geq r_2}^{\ m_2}, \dots,  B_{\geq r_1}^{\
m_l}\subseteq B_{\geq r_l}^{\ m_l}$; from this we get that $B_{\geq r_1}^{\ m_1}a_1x=0, B_{\geq
r_1}^{\ m_2}a_2x=0,\dots, B_{\geq r_1}^{\ m_l}a_lx=0$, let $m:=\max\{m_1,\dots,m_l\}$, then
$B_{\geq r}^{\ m}a_ix=0$ for every $1\leq i\leq l$, with $r:=r_1$. Therefore, $B_{\geq r}^{\
m}B_{\geq s}^{\ n}x=B_{\geq r}^{\ m}(Ba_1+\cdots+Ba_l)x= B_{\geq r}^{\ m}a_1x+\cdots +B_{\geq r}^{\
m}a_lx=0$, i.e., $B_{\geq r}^{\ m}B_{\geq s}^{\ n}x=0$. Observe that $B_{\geq r+s}^{\ m}\subseteq
B_{\geq r}^{\ m}$ and $B_{\geq r+s}^{\ n}\subseteq B_{\geq s}^{\ n}$, so $B_{\geq r+s}^{\
m+n}\subseteq B_{\geq r}^{\ m}B_{\geq s}^{\ n}$ and hence $B_{\geq r}^{\ m+n}x=0$, i.e., $x\in
T(M)$. We have proved that $T(M)=M$, i.e., $M$ is in ${\rm stor}-B$.

The second statement of the theorem is a well known property of abelian categories. We want to
recall that the objects of ${\rm qsgr}-B$ are the objects of ${\rm sgr}-B$; moreover, given $M,N$
objects of ${\rm qsgr}-B$ the set of morphisms from $M$ to $N$ in the category ${\rm qsgr}-B$ is
defined by
\begin{center}
$Hom_{\rm {qsgr}-B}(M,N):=\underrightarrow{\lim}Hom_{{\rm sgr}-B}(M', N/N')$,
\end{center}
where the direct limit is taken over all $M'\subseteq M$, $N'\subseteq N$ in ${\rm sgr}-B$ with
$M/M'\in {\rm stor}-B$ and $N'\in {\rm stor}-B$ (see \cite{Grothendieck}, \cite{Gabriel}, or also
\cite{Smith} Proposition 2.13.4 ). More exactly, the limit is taken over the set $\mathcal{P}$ of
all pairs $(M',N')$ in ${\rm sgr}-B$ such that $ M'\subseteq M$, $N'\subseteq N$, $M/M'\in {\rm
stor}-B$ and $N'\in {\rm stor}-B$. The set $\mathcal{P}$ is partially ordered with order defined by
\begin{center}
$(M',N')\leq (M'',N'')$ if and only if $M''\subseteq M'$ and $N'\subseteq N''$.
\end{center}
$\mathcal{P}$ is directed: Indeed, given $(M',N'),(M'',N'')\in \mathcal{P}$ we apply Proposition
\ref{proposition17.5.5} and the fact that $B$ is left noetherian to conclude that $(M'\cap
M'',N'+N'')\in \mathcal{P}$, and this couple satisfies $(M',N')\leq (M'\cap M'',N'+N'')$,
$(M'',N'')\leq (M'\cap M'',N'+N'')$.

\end{proof}

We have all ingredients in order to define non-commutative schemes associated to semi-graded rings.

\begin{definition}\label{definition16.5.34}
We define
\begin{center}
${\rm sproj}(B):=({\rm qsgr}-B,\pi(B))$
\end{center}
and we call it the non-commutative semi-projective scheme associated to $B$.
\end{definition}

\section{Serre-Artin-Zhang-Verevkin theorem for semi-graded rings}

We conclude the paper investigating the non-commutative version of Serre-Artin-Zhang-Verevkin
theorem for semi-graded rings. For this goal some preliminaries are needed.

\begin{definition}
Let $M$ be a semi-graded $B$-module, $M=\bigoplus_{n\in \mathbb{Z}} M_n$. Let $i\in \mathbb{Z}$,
the semi-graded module $M(i)$ defined by $M(i)_{n}:=M_{i+n}$ is called a shift of $M$, i.e.,
\begin{center}
$M(i)=\bigoplus_{n\in \mathbb{Z}}M(i)_n=\bigoplus_{n\in \mathbb{Z}} M_{i+n}$.
\end{center}
\end{definition}

\begin{remark}
Note that for every $i\in \mathbb{Z}$, $M\cong M(i)$ as $B$-modules. The isomorphism is given by
\begin{center}
$M=\bigoplus_{n\in \mathbb{Z}} M_n \xrightarrow{\phi_i} \bigoplus_{n\in \mathbb{Z}} M_{i+n}=M(i)$

$m_{n_1}+\cdots+m_{n_t}\in M_{n_1}+\cdots+M_{n_t}\mapsto m_{n_1}+\cdots+m_{n_t}\in
M(i)_{n_1-i}+\cdots+M(i)_{n_t-i}$.
\end{center}
$\phi_i$ is not homogeneous for $i\neq 0$.
\end{remark}

The next proposition shows that the shift of degrees is an autoequivalence.

\begin{proposition}
Let $s:{\rm sgr}-B\to {\rm sgr}-B$ defined by
\begin{center}
$M\mapsto M(1)$,

$M\xrightarrow{f} N\mapsto M(1)\xrightarrow{f(1)} N(1)$,

$f(1)(m):=f(m)$, $m\in M(1)$.
\end{center}
Then,
\begin{enumerate}
\item[\rm (i)]$s$ is an autoequivalence.
\item[\rm (ii)]For every $d\in \mathbb{Z}$, $s^d(M)=M(d)$.
\item[\rm (iii)]$s$ induces an autoequivalence of ${\rm qsgr}-B$ also denoted by $s$.
\end{enumerate}
\end{proposition}
\begin{proof}
(i) and (ii) are evident. For (iii) we only have to observe that if $M$ is an object of ${\rm
stor}-B$, then $s(M)$ is also an object of ${\rm stor}-B$ .
\end{proof}

\begin{definition}
Let $M,N$ be objects of ${\rm sgr}-B$. Then
\begin{enumerate}
\item[\rm (i)]$\underline{Hom}_B(M,N):=\bigoplus_{d\in \mathbb{Z}}Hom_{{\rm sgr}-B}(M,N(d))$.
\item[\rm (ii)]$\underline{Ext}_B^i(M,N):=\bigoplus_{d\in \mathbb{Z}}Ext^i_{{\rm sgr}-B}(M,N(d))$.
\end{enumerate}
\end{definition}

\begin{remark}\label{proposition16.5.39}
Note that $\underline{Hom}_B(M,N)\hookrightarrow Hom_B(M,N)$. In fact, we have the group
homomorphism $\iota: \underline{Hom}_B(M,N)\to Hom_B(M,N)$ given by $(\dots, 0,f_{d_1},\dots,
f_{d_t}, 0, \dots)\mapsto f_{d_1}+\cdots +f_{d_t}$; observe that $f_{d_1}+\cdots +f_{d_t}=0$ if and
only if $f_{d_1}=\cdots =f_{d_t}=0$. Indeed, let $m\in M$ be homogeneous of degree $p$, then
$0=(f_{d_1}+\cdots +f_{d_t})(m)=f_{d_1}(m)+\cdots+f_{d_t}(m)\in N_{d_1+p}\oplus \cdots \oplus
N_{d_t+p}$, whence, for every $j$, $f_{d_j}(m)=0$. This means that $f_{d_j}=0$ for $1\leq j\leq t$,
and hence, $\iota$ is injective.
\end{remark}

\begin{proposition}\label{proposition16.5.40}
Let $M$ and $N$ be semi-graded $B$-modules such that every of its homogeneous components are
$B_0$-modules. Then,
\begin{enumerate}
\item[\rm (i)]$Hom_{{\rm sgr}-B}(M,N)$ is a $B_0$-module.
\item[\rm (ii)]$\underline{Hom}_{B}(M,N)$ is a $B_0$-module.
\item[\rm (iii)]$Ext_{{\rm sgr}-B}^i(M,N)$ is a $B_0$-module for every $i\geq 1$.
\item[\rm (iv)]$\underline{Ext}_{B}^i(M,N)$ is a $B_0$-module for every $i\geq 1$.
\end{enumerate}
\end{proposition}
\begin{proof}
(i) If $f\in Hom_{{\rm sgr}-B}(M,N)$ and $b_0\in B_0$, then product $b_0\cdot f$ defined by
$(b_0\cdot f)(m):=b_0\cdot f(m)$, $m\in M$, is an element of $Hom_{{\rm sgr}-B}(M,N)$: In fact,
$b_0\cdot f$ is obviously additive; let $b\in B$, then $(b_0\cdot f)(b\cdot m)=b_0\cdot f(b\cdot
m)=b_0[b\cdot f(m)]= (b_0b)\cdot f(m)=(bb_0)\cdot f(m)=b\cdot (b_0\cdot f(m))=b\cdot (b_0\cdot
f)(m)$; $b_0\cdot f$ is homogeneous: Let $m\in M_p$, then $(b_0\cdot f)(m)=b_0\cdot f(m)\in
b_0\cdot N_p\subseteq N_p$, for every $p\in\mathbb{Z}$. It is easy to check that  $Hom_{{\rm
sgr}-B}(M,N)$ is a $B_0$-module with the defined product.

(ii) This follows from (i).

(iii) Taking a projective resolution of $M$ in the abelian category ${\rm sgr}-B$ and applying the
functor $Hom_{{\rm sgr}-B}(-,N)$, it is easy to verify using (i) that in the complex defining
$Ext_{{\rm sgr}-B}^i(M,N)$ the kernels and the images are $B_0$-modules, i.e., every abelian group
$Ext_{{\rm sgr}-B}^i(M,N)$ is a $B_0$-module.

(iv) This follows from (iii).
\end{proof}

\begin{definition}
Let $i\geq 0$; we say that $B$ satisfies the s-$\chi_i$ condition if for every finitely generated
semi-graded $B$-module $N$ and for any $j\leq i$, $\underline{Ext}_B^j(B/B_{\geq 1},N)$ is finitely
generated as $B_0$-module. The ring $B$ satisfies the s-$\chi$ condition if it satisfies the
s-$\chi_i$ condition for all $i\geq 0$.
\end{definition}

\begin{remark}
(i) By Proposition \ref{proposition16.5.40}, $\underline{Ext}_B^j(B/B_{\geq 1},N)$ is a
$B_0$-module.

(ii) In the theory of graded rings and modules the conditions defined above are usually denoted
simply by $\chi_i$ and $\chi$. In this situation, $B/B_{\geq 1}\cong B_0$.

(iii) Observe that in the case of finitely graded $K$-algebras, $B_0=K$, $B/B_{\geq 1}\cong K$ and
the the condition $s-\chi_i$ means that ${\rm dim}_{K}\underline{Ext}_B^j(K,N)<\infty$.
\end{remark}

\begin{definition}
Let $s$ be the autoequivalence of ${\rm qsgr}-B$ defined by the shifts of degrees. We define
\begin{center}
$\Gamma(\pi(B))_{\geq 0}:=\bigoplus_{d=0}^\infty Hom_{{\rm qsgr}-B}(\pi(B),s^d(\pi(B)))$.
\end{center}
\end{definition}

Following the ideas in the proof of Theorem 4.5 in \cite{Artin2} and Proposition 4.11 in
\cite{Rogalski} we get the following key lemma.

\begin{lemma}\label{lemma6.11}
Let $B$ be a ring that satisfies {\rm (C1)-(C4)}.
\begin{enumerate}
\item[\rm (i)]$\Gamma(\pi(B))_{\geq 0}$ is a $\mathbb{N}$-graded ring.
\item[\rm (ii)]Let $\underline{B}:=\bigoplus_{d=0}^\infty Hom_{{\rm sgr}-B}(B,s^d(B))$.
Then, $\underline{B}$ is a $\mathbb{N}$-graded ring and there exists a ring homomorphism
$\underline{B}\to \Gamma(\pi(B))_{\geq 0}$.
\item[\rm (iii)]For any object $M$ of ${\rm sgr}-B$
\begin{center}
$\Gamma(M)_{\geq 0}:=\bigoplus_{d=0}^\infty Hom_{{\rm sgr}-B}(B,s^d(M))$
\end{center}
is a graded $\underline{B}$-module, and
\begin{center}
$\Gamma(\pi(M))_{\geq 0}:=\bigoplus_{d=0}^\infty Hom_{{\rm qsgr}-B}(\pi(B),s^d(\pi(M)))$
\end{center}
is a graded $\Gamma(\pi(B))_{\geq 0}$-module.
\item[\rm (iv)]$\underline{B}$ has the following properties:
\begin{enumerate}
\item[\rm (a)]$(\underline{B})_0\cong B_0$ and $\underline{B}$ satisfies {\rm (C2)}.
\item[\rm (b)]$\underline{B}$ satisfies {\rm (C3)}. More generally, let $N$ be a finitely generated graded
$\underline{B}$-module, then every homogeneous component of $N$ is finitely generated over
$(\underline{B})_0$.
\item[\rm (c)]$\underline{B}$ satisfies {\rm (C1)}.
\item[\rm (d)]If $B$ is a domain, then $\underline{B}$ is also a domain.
\end{enumerate}
\item[\rm (v)] If $B$ is a domain, then
\begin{enumerate}
\item[\rm (a)]$\Gamma(\pi(B))_{\geq 0}$ satisfies {\rm (C2)}.
\item[\rm (b)]$\Gamma(\pi(B))_{\geq 0}$ satisfies {\rm (C3)}. More generally, let $N$ be a finitely generated graded
$\Gamma(\pi(B))_{\geq 0}$-module, then every homogeneous component of $N$ is finitely generated
over $(\Gamma(\pi(B))_{\geq 0})_0$.
\item[\rm (c)]$\Gamma(\pi(B))_{\geq 0}$ satisfies {\rm (C1)}.
\item[\rm (d)]If $\underline{B}$ satisfies $\mathcal{X}_1$, then $\Gamma(\pi(B))_{\geq 0}$ satisfies $\mathcal{X}_1$.
\item[\rm (e)]$\Gamma(\pi(B))_{\geq 0}$ is a domain.
\end{enumerate}
\end{enumerate}
\end{lemma}
\begin{proof}
(i) Since ${\rm qsgr}-B$ is an abelian category, $Hom_{{\rm qsgr}-B}(\pi(B),s^d(\pi(B)))$ is an
abelian group; the product in $\Gamma(\pi(B))_{\geq 0}$ is defined by distributive law and the
following rule:
\begin{center}
If $f\in Hom_{{\rm qsgr}-B}(\pi(B),s^n(\pi(B)))$ and $g\in Hom_{{\rm qsgr}-B}(\pi(B),s^m(\pi(B)))$,
then

$f\star g:=s^n(g)\circ f\in Hom_{{\rm qsgr}-B}(\pi(B),s^{m+n}(\pi(B)))$.
\end{center}
This product is associative: In fact, if $h\in Hom_{{\rm qsgr}-B}(\pi(B),s^{p}(\pi(B)))$, then
\begin{center}
$(f\star g)\star h=[s^n(g)\circ f]\star h=s^{m+n}(h)\circ s^n(g)\circ f=f\star (g\star h)$.
\end{center}
It is clear that the product is $\mathbb{N}$-graded and the unity of $\Gamma(\pi(B))_{\geq 0}$ is
$i_B$ taken in $d=0$ (observe that we have simplified the notation avoiding the bar notation for
the morphisms in the category ${\rm qsgr}-B$).

(ii) The proof of that $\underline{B}$ is a $\mathbb{N}$-graded ring is as in (i). For the second
assertion we can apply the quotient functor $\pi$ to define the function
\begin{align}\label{equation6.1}
\underline{B} & \xrightarrow{\rho} \Gamma(\pi(B))_{\geq 0} \\
(f_{0},\dots, f_{d}, 0, \dots) & \mapsto (\pi(f_{0}),\dots, \pi(f_{d}), 0, \dots) \notag
\end{align}
which is a ring homomorphism since $\pi$ is additive ($\pi$ is exact) and $s\pi=\pi s$.

(iii) The proof of both assertions are as in (i), we only illustrate the product in the first case:
\begin{center}
If $f\in Hom_{{\rm sgr}-B}(B,s^n(B))$ and $g\in Hom_{{\rm sgr}-B}(B,s^m(M))$, then

$f\star g:=s^n(g)\circ f\in Hom_{{\rm sgr}-B}(B,s^{m+n}(M))$.
\end{center}

(iv) (a) Note that $(\underline{B})_0=Hom_{{\rm sgr}-B}(B,B)$, and consider the function
\begin{center}
$B_0\xrightarrow{\alpha} Hom_{{\rm sgr}-B}(B,B)$, $\alpha(x)=\alpha_x$, $\alpha_x(b):=bx$, $x\in
B_0$, $b\in B$;
\end{center}
since $B_0\subset Z(B)$ this function is a ring homomorphism, moreover, bijective. Thus,
$(\underline{B})_0$ is a commutative noetherian ring, so $(\underline{B})_0$ satisfies (C2). In
addition, observe that the structure of $B_0$-module of $Hom_{{\rm sgr}-B}(B,B)$ induced by
$\alpha$ coincides with the structure defined in Proposition \ref{proposition16.5.40}.

(b) Note that the function $Hom_{{\rm sgr}-B}(B,B(d))\xrightarrow{\lambda}B_d$ defined by $f\mapsto
f(1)$ is an injective $B_0$-homomorphism. Since $B_0$ is noetherian and $B$ satisfies {\rm C3},
then $Hom_{{\rm sgr}-B}(B,B(d))$ is finitely generated over $B_0\cong (\underline{B})_0$.

For the second part, let $N$ be generated by  $x_1,\dots,x_r$, with $x_i\in N_{d_i}$, $1\leq i\leq
r$. Let $x\in N_d$, then there exist $f_1,\dots,f_r\in \underline{B}$ such that $x=f_1\cdot
x_1+\cdots+f_r\cdot x_r$, from this we can assume that $f_i\in (\underline{B})_{d-d_i}$; by the
just proved property (C3) for $\underline{B}$ we obtain that every $(\underline{B})_{d-d_i}$ is
finitely generated as $(\underline{B})_{0}$-module, this implies that $N_d$ is finitely generated
over $(\underline{B})_{0}$.

(c) By (ii), $\underline{B}$ is not only $SG$ but $\mathbb{N}$-graded.

$\underline{B}$ is left noetherian: We will adapt a proof given in \cite{Artin2}. Let $I$ be a
graded left ideal of $\underline{B}$; let $f\in \underline{B}$ be homogeneous of degree $d_f$, then
$f$ induces a morphism $s^{-d_f}(B)\xrightarrow{f_-} B$; thus, given a finite set $F$ of
homogeneous elements of $I$, let $P_F:=\bigoplus_{f\in F} s^{-d_f}(B)$, $f_F:=\sum_{f\in
F}f_-:P_F\rightarrow B$ and let $N_F:=Im(f_F)$. Since $B$ is left noetherian we can choose a finite
set $F_0$ such that $N_{F_0}$ is maximal among such images. Let $N:=N_{F_0}$ and $P:=P_{F_0}$; we
define $N'':=\Gamma(N)_{\geq 0}:=\bigoplus_{d=0}^\infty Hom_{{\rm sgr}-B}(B,s^d(N))$. According to
(iii), $N''$ is a $\mathbb{N}$-graded $\underline{B}$-module. Given any element $f\in I$
homogeneous of degree $d_f$ we have the morphism $f_-$, but since $N$ is maximal the image of this
morphism is included in $N$, and this implies that $f\in N''$, so $I\subseteq N''$. On the other
hand, given $f\in I$ homogeneous of degree $d_f$ the $\mathbb{N}$-graded
$\underline{B}$-homomorphism $s^{-d_f}(\underline{B})\xrightarrow{\textit{\textbf{f}}_-}
\underline{B}$ defined by $\textit{\textbf{f}}_-(h):=hf$ has his image in $I$. Therefore,
$N'\subseteq I$, where $N'$ is the image of the induced morphism $P''\to \underline{B}$, with
$P'':=\bigoplus_{f\in F_0} s^{-d_f}(\underline{B})$. Thus, we have $N'\subseteq N''$, where both
are $\mathbb{N}$-graded $\underline{B}$-modules, whence we have the $\mathbb{N}$-graded
$\underline{B}$-module $N''/N'$. If we prove that $N''/N'$ is noetherian, then since $I/N'\subseteq
N''/N'$ we get that $I/N'$ is also noetherian, whence, $I/N'$ is finitely generated; but $N'$ is a
finitely generated left ideal of $\underline{B}$, so $I$ is finitely generated.

$N''/N'$ is noetherian:  Note first that $N''/N'$ is a module over $(\underline{B})_0$; if we prove
that $N''/N'$ is noetherian over $(\underline{B})_0$, then it is also noetherian over
$\underline{B}$. According to (a), we only need to show that $N''/N'$ is finitely generated over
$(\underline{B})_0$. But this follows from (b) since $N''/N'$ is right bounded (i.e., there exists
$n\gg 0$ such that the homogeneous component of $N''/N'$ of degree $k\geq n$ is zero, see
\cite{Artin2}).

(d) If $B$ is a domain, then $\underline{B}$ is also a domain: Suppose there exist $f,g\neq 0$ in
$\underline{B}$ such that $f\star g=0$, let $f_n\neq 0$ and $g_m\neq 0$ the nonzero homogeneous
components of $f$ and $g$ of lowest degree, thus $f_n\in Hom_{{\rm sgr}-B}(B,s^n(B))$, $g_m\in
Hom_{{\rm sgr}-B}(B,s^m(B))$ and $0=f_n\star g_m=s^n(g_m)\circ f_n\in Hom_{{\rm
sgr}-B}(B,s^{m+n}(B))$; since $f_n\neq 0$ we have $f_n(1)\neq 0$, also $g_m(1)\neq 0$ and hence
$s^n(g_m)(1)\neq 0$, so $0=s^n(g_m)(f_n(1))=f_n(1)s^n(g_m)(1)$,  but this is impossible since $B$
is a domain.

(v) We set $\Gamma:= \Gamma(\pi(B))_{\geq 0}$. Then,

(a) $\Gamma$ satisfies (C2): We divide the proof of this statement in two steps.

\textit{Step 1}. Adapting the proof of Proposition 5.3.7 in \cite{Smith} we will show that
\begin{center}
$\Gamma_0=Hom_{{\rm qsgr}-B}(\pi(B),\pi(B))=\underrightarrow{\lim}Hom_{{\rm sgr}-B}(B_{s,n}, B)$,
\end{center}
where the direct limit is taken over the homomorphisms of abelian groups
\begin{center}
$Hom_{{\rm sgr}-B}(B_{s,n}, B)\to Hom_{{\rm sgr}-B}(B_{r,m}, B)$
\end{center}
induced by the inclusion homomorphism $B_{r,m}\to B_{s,n}$, with $r\geq s$ and $m\geq n$. Observe
that the collection of couples $(s,n)$ is a partially ordered directed set.

Note first that $Hom_{{\rm qsgr}-B}(\pi(B),\pi(B))=\underrightarrow{\lim}Hom_{{\rm sgr}-B}(M', B)$,
where the direct limit is taken over all $M'\subseteq B$ with $B/M'\in {\rm stor}-B$. In fact, we
know that
\begin{center}
$Hom_{\rm {qsgr}-B}(\pi(B),\pi(B))=\underrightarrow{\lim}Hom_{{\rm sgr}-B}(M', B/N')$,
\end{center}
where the direct limit is taken over all $(M',N')\in \mathcal{P}$, but since $B$ is a domain,
$N'=0$.

Now let $\overline{f}\in \underrightarrow{\lim}Hom_{\rm {sgr}-B}(M',B)$, so $f\in Hom_{\rm
{sgr}-B}(M',B)$ for some $M'\subseteq B$ such that $T(B/M')=B/M'$; since $B/M'$ is finitely
generated, we can reasoning as in the proof of Theorem \ref{theorem16.6.5} and find $s,n\geq 0$
such that $B_{s}^nB\subseteq M'$, i.e., $B_{s,n}\subseteq M'$. From this we get that
$\overline{f}=\overline{f'}$, where $\overline{f'}\in  Hom_{{\rm sgr}-B}(B_{s,n}, B)$, with
$f':=f\iota$ and $\iota:B_{s,n}\hookrightarrow M'$ the inclusion. Since $B/B_{s,n}$ is torsion we
obtain that $\underrightarrow{\lim}Hom_{{\rm sgr}-B}(M', B)=\underrightarrow{\lim}Hom_{{\rm
sgr}-B}(B_{s,n}, B)$.

\textit{Step 2}. Considering $s,n=0$ in the limit above we obtain a ring homomorphism
\begin{center}
$(\underline{B})_0=Hom_{{\rm sgr}-B}(B, B)\xrightarrow{\gamma} Hom_{{\rm
qsgr}-B}(\pi(B),\pi(B))=\Gamma_0$;
\end{center}
since $(\underline{B})_0$ is noetherian we can prove that $\gamma$ is surjective. Let
$\overline{f}\in \Gamma_0$ with $f\in Hom_{{\rm sgr}-B}(B_{s,n},B)$, consider  the commutative
triangles
\[
\begin{diagram}
\node{} \node{B} \node{}\\
\node{B_{s,n}} \arrow{ne,r}{f} \arrow[2]{e,b}{\iota} \node{} \node{B_{0,0}=B} \arrow{nw,r}{f'}
\end{diagram} \ \ \begin{diagram}
\node{} \node{B} \node{}\\
\node{B_{s,n}} \arrow{ne,r}{f} \arrow[2]{e,b}{i} \node{} \node{B_{s,n}} \arrow{nw,r}{f},
\end{diagram}
\]
where $f'$ is defined by $f'(x+l):=f(x)$, with $x\in B_{s,n}$, $l\in L$ and $B=B_{s,n}\oplus L$.
Thus, $i^*(f)=fi=f$ and $\iota^*(f')=f'\iota=f$, so $\overline{f}=\overline{f'}=\gamma(f')$.

From this we conclude that $\Gamma_0$ is a commutative noetherian ring, and hence, $\Gamma$
satisfies (C2).

(b) $\Gamma$ satisfies (C3): Since $\Gamma$ is graded, $\Gamma_d$ is a $\Gamma_0$-module for every
$d$, but by (a) we have a ring homomorphism $B_0\cong (\underline{B})_0\xrightarrow{\gamma}
\Gamma_0$, so  the idea is to prove that $\Gamma_d$ is finitely generated over $B_0$. For this we
will show that there exists a surjective $B_0$-homomorphism
$(\underline{B})_d\xrightarrow{\beta}\Gamma_d$. Note that $\Gamma_d=Hom_{{\rm
qsgr}-B}(\pi(B),\pi(B(d)))= \underrightarrow{\lim}Hom_{{\rm sgr}-B}(B_{s,n}, B(d))$ (the proof of
this is as the step 1 in (a)); let $f\in (\underline{B})_d=Hom_{{\rm sgr}-B}(B, B(d))$, we define
$\beta(f):=\overline{f\iota}$, where $\iota:B_{s,n}\to B=B_{0,0}$; we can repeat the proof of the
step 2 in (a) and conclude that $\beta$ is a surjective $B_0$-homomorphism.

Additionally, let $N$ be a finitely generated graded $\Gamma$-module, says $N$ generated by a
finite set of homogeneous elements $x_1,\dots,x_r$, with $x_i\in N_{d_i}$, $1\leq i\leq r$. Let
$x\in N_d$, then there exist $f_1,\dots,f_r\in \Gamma$ such that $x=f_1\cdot x_1+\cdots+f_r\cdot
x_r$, from this we can assume that $f_i\in \Gamma_{d-d_i}$, but as was observed before, every
$\Gamma_{d-d_i}$ is finitely generated as $\Gamma_{0}$-module, so $N_d$ is finitely generated over
$\Gamma_{0}$ for every $d$.

(c) $\Gamma$ satisfies (C1): By (iii), $\Gamma$ is not only $SG$ but $\mathbb{N}$-graded.

$\Gamma$ is left noetherian: We will adapt the proof of (iv)-(c). Let $I$ be a graded left ideal of
$\Gamma$; let $f\in \Gamma$ be homogeneous of degree $d_f$, then $f$ induces a morphism
$s^{-d_f}(\pi(B))\xrightarrow{f_-} \pi(B)$; thus, given a finite set $F$ of homogeneous elements of
$I$, let $P_F:=\bigoplus_{f\in F} s^{-d_f}(\pi(B))$, $f_F:=\sum_{f\in F}f_-:P_F\rightarrow \pi(B)$
and let $N_F:=Im(f_F)$. Since $\pi(B)$ is a noetherian object of ${\rm qsgr}-B$ we can choose a
finite set $F_0$ such that $N_{F_0}$ is maximal among such images. Let $\pi(N):=N_{F_0}$ and
$\pi(P):=P_{F_0}$; we define $N'':=\Gamma(\pi(N))_{\geq 0}:=\bigoplus_{d=0}^\infty Hom_{{\rm
qsgr}-B}(\pi(B),s^d(\pi(N)))$. According to (iii), $N''$ is a $\mathbb{N}$-graded $\Gamma$-module.
Given any element $f\in I$ homogeneous of degree $d_f$ we have the morphism $f_-$, but since $N$ is
maximal the image of this morphism is included in $N$, and this implies that $f\in N''$, so
$I\subseteq N''$. On the other hand, given $f\in I$ homogeneous of degree $d_f$ the
$\mathbb{N}$-graded $\Gamma$-homomorphism $s^{-d_f}(\Gamma)\xrightarrow{\textit{\textbf{f}}_-}
\Gamma$ defined by $\textit{\textbf{f}}_-(h):=hf$ has his image in $I$. Therefore, $N'\subseteq I$,
where $N'$ is the image of the induced morphism $P''\to \Gamma$, with $P'':=\bigoplus_{f\in F_0}
s^{-d_f}(\Gamma)$. Thus, we have $N'\subseteq N''$, where both are $\mathbb{N}$-graded
$\Gamma$-modules, whence we have the $\mathbb{N}$-graded $\Gamma$-module $N''/N'$. If we prove that
$N''/N'$ is noetherian, then since $I/N'\subseteq N''/N'$ we get that $I/N'$ is also noetherian,
whence, $I/N'$ is finitely generated; but $N'$ is a finitely generated left ideal of $\Gamma$, so
$I$ is finitely generated.

$N''/N'$ is noetherian:  Note first that $N''/N'$ is a module over $\Gamma_0$; if we prove that
$N''/N'$ is noetherian over $\Gamma_0$, then it is also noetherian over $\Gamma$. According to (a),
we only need to show that $N''/N'$ is finitely generated over $\Gamma_0$. But this follows from (b)
since $N''/N'$ is right bounded.

(d) $\Gamma$ satisfies $\mathcal{X}_1$: Let $N$ be a finitely generated graded $\Gamma$-module, we
have $\underline{Ext}_{\Gamma}^j(\Gamma/\Gamma_{\geq 1},N)=
\underline{Ext}_{\Gamma}^j(\Gamma_0,N)$, so we must prove that
$\underline{Ext}_{\Gamma}^j(\Gamma_0,N)$ is finitely generated as $\Gamma_0$-module for $j=0,1$. By
the surjective homomorphism $(\underline{B})_0\xrightarrow{\gamma}\Gamma_0$ in the step 2 in (a),
it is enough to show that $\underline{Ext}_{\Gamma}^j(\Gamma_0,N)$ is finitely generated over
$(\underline{B})_0$. Observe that $\gamma$ is also a graded homomorphism of left
$(\underline{B})$-modules; moreover, $N$ is a finitely generated graded left
$(\underline{B})$-module since the homomorphism $\rho$ in (ii) is surjective; the proof of this
last statement is as in the step 2 of (a), using of course that $B$ is a domain, we include it for
completeness: It is enough to consider $\overline{f}\in \Gamma_d= \underrightarrow{\lim}Hom_{{\rm
sgr}-B}(B_{s,n}, B(d))$, with $f\in Hom_{{\rm sgr}-B}(B_{s,n}, B(d))$ for some $s,n\geq 0$; we
define $f':B_{0,0}\to B(d)$, $f'(x):=f(y)$, where $B=B_{s,n}\oplus L$ and $x=y+l$ with $y\in
B_{s,n}$ and $l\in L$; therefore, $\rho(f')=\pi(f')=\overline{f}$ since we have $f'\iota=f$.

Now we can apply the functor $\underline{Ext}_{\underline{B}}^j(\, ,N)$ and get the injective
homomorphism of left $(\underline{B})_0$-modules $\underline{Ext}_{\underline{B}}^j(\Gamma_0,N)\to
\underline{Ext}_{\underline{B}}^j((\underline{B})_0,N)$, but since $\underline{B}$ satisfies
$\mathcal{X}_1$, $\underline{Ext}_{\underline{B}}^j((\underline{B})_0,N)$ is finitely generated
over $(\underline{B})_0$, so $\underline{Ext}_{\underline{B}}^j(\Gamma_0,N)$ is finitely generated
since $(\underline{B})_0$ is left noetherian. From the injective $(\underline{B})_0$-homomorphism
$\underline{Ext}_{\Gamma}^j(\Gamma_0,N)\to \underline{Ext}_{\underline{B}}^j(\Gamma_0,N)$ we
conclude that $\underline{Ext}_{\Gamma}^j(\Gamma_0,N)$ is also finitely generated over
$(\underline{B})_0$.

(e) $\Gamma$ is a domain: Suppose there exist $f,g\neq 0$ in $\Gamma$ such that $f\star g=0$, let
$f_n\neq 0$ and $g_m\neq 0$ the nonzero homogeneous components of $f$ and $g$ of lowest degree,
thus
\begin{center}
$f_n\in Hom_{{\rm qsgr}-B}(\pi(B),s^n(\pi(B)))$, $g_m\in Hom_{{\rm qsgr}-B}(\pi(B),s^m(\pi(B)))$
\end{center}
and $0=f_n\star g_m=s^n(g_m)\circ f_n\in Hom_{{\rm qsgr}-B}(\pi(B),s^{m+n}(\pi(B)))$; note that the
representative elements of $f_n$ and $g_m$ in $Hom_{{\rm sgr}-B}(B_{0,0},B(n))\cong Hom_{{\rm
sgr}-B}(B,B)=(\underline{B})_0$ and $Hom_{{\rm sgr}-B}(B_{0,0},B(m))\cong Hom_{{\rm
sgr}-B}(B,B)=(\underline{B})_0$, respectively, are non zero, but this is impossible since
$(\underline{B})_0$ is a domain and the representative element of $f\star g$ in $Hom_{{\rm
sgr}-B}(B_{0,0},B(n+m))\cong Hom_{{\rm sgr}-B}(B,B)=(\underline{B})_0$ is zero.
\end{proof}

\begin{proposition}\label{proposition6.12}
Let $S$ be a commutative noetherian ring and $\rho:C\to D$ be a homomorphism of $\mathbb{N}$-graded
left noetherian $S$-algebras. If the kernel and cokernel of $\rho$ are right bounded, then
$D\underline{\otimes}\,_C\, -$ defines an equivalence of categories ${\rm qgr}-C\simeq {\rm
qgr}-D$, where $\underline{\otimes}$ denotes the graded tensor product.
\end{proposition}
\begin{proof}
The proof of Proposition 2.5 in \cite{Artin2} applies since it is independent of the notion of
torsion.
\end{proof}

We are prepared for proving the main theorem of the present section.

\begin{theorem}\label{theorem16.7.13}
If $B$ is a domain that satisfies {\rm (C1)-(C4)} and $\underline{B}$ satisfies the condition
$\mathcal{X}_1$ then there exists an equivalence of categories
\begin{center}
${\rm qgr}-\underline{B}\simeq{\rm qgr}-\Gamma(\pi(B))_{\geq 0}$.
\end{center}
\end{theorem}
\begin{proof}
Note that the ring homomorphism in (\ref{equation6.1}) satisfies the conditions of Proposition
\ref{proposition6.12}, with $S=B_0$, $C=\underline{B}$ and $D=\Gamma(\pi(B))_{\geq 0}$. In fact,
from Lemma \ref{lemma6.11} we know that $\underline{B}$ and $\Gamma(\pi(B))_{\geq 0}$ are
$\mathbb{N}$-graded left noetherian rings and $B_0$-modules; moreover, they are $B_0$-algebras: We
check this for $\underline{B}$, the proof for $\Gamma(\pi(B))_{\geq 0}$ is similar. If $f\in
Hom_{{\rm sgr}-B}(B,B(n)), g\in Hom_{{\rm sgr}-B}(B,B(m))$, $x\in B_0$ and $b\in B$, then
\begin{center}
$[x\cdot (f\star g)](b)=x\cdot (s^n(g)\circ f)(b)=xg(n)(f(b))$;

$[f\star(x\cdot g)](b)=[s^n(x\cdot g)\circ f](b)=(x\cdot g)(n)(f(b))=xg(n)(f(b))$.
\end{center}
Finally, we can apply the proof of part S10 in Theorem 4.5 in \cite{Artin2} to conclude that the
kernel and cokernel of $\rho$ are right bounded.
\end{proof}

\begin{remark}
Considering the above developed theory for graded rings and right modules instead of semi-graded
rings and left modules it is possible to prove that $\underline{B}\cong B$. Thus, in such case we
get from the previous theorem the Artin-Zhang-Verevkin equivalence ${\rm qgr}-B\simeq{\rm
qgr}-\Gamma(\pi(B))_{\geq 0}$.
\end{remark}

\begin{example}\label{example6.14}
The examples of skew $PBW$ extensions below are semi-graded (non $\mathbb{N}$-graded) domains and
satisfy the conditions (C1)-(C4); in each case we will prove that $\underline{B}$ satisfies the
condition $\mathcal{X}_1$; therefore, for these algebras Theorem \ref{theorem16.7.13} is true. In
every example $B_0=K$ is a field, we indicate the relations defining $B$ (see \cite{lezamareyes1})
and the associated graded ring $Gr(B)$ (Proposition \ref{proposition17.5.5}):

(i) Enveloping algebra of a Lie $K$-algebra $\mathcal{G}$ of dimension $n$, $\cU(\mathcal{G})$:
\begin{center}
$x_ik-kx_i=0$, $k\in K$;

$x_ix_j-x_jx_i=[x_i,x_j]\in \mathcal{G}=Kx_1+\cdots+Kx_n$, $1\leq i,j\leq n$;

$Gr(B)=K[x_1,\dots,x_n]$.
\end{center}

(ii) Quantum algebra $\mathcal{U}'(so(3,K))$, with $q\in K-\{0\}$:
\begin{center}
$x_2x_1-qx_1x_2=-q^{1/2}x_3,\ \ \ x_3x_1-q^{-1}x_1x_3=q^{-1/2}x_2,\ \ \
x_3x_2-qx_2x_3=-q^{1/2}x_1$;
\end{center}
in this case $Gr(B)=K_{\textbf{q}}[x_1,x_2,x_3]$ is the $3$-multiparametric quantum space, i.e., a
quantum polynomial ring in $3$ variables, with
\begin{center}
$\textbf{q}=\begin{bmatrix}1 & q & q^{-1}\\
q^{-1} & 1 & q\\
q & q^{-1} & 1\end{bmatrix}$.
\end{center}

(iii) Dispin algebra $\cU(osp(1,2))$:
\begin{center}
$x_2x_3-x_3x_2=x_3,\ \ \ x_3x_1+x_1x_3=x_2,\ \ \ x_1x_2-x_2x_1=x_1$;

$Gr(B)=K_{\textbf{q}}[x_1,x_2,x_3]$, with $\textbf{q}=\begin{bmatrix}1 & 1 & -1\\
1 & 1 & 1\\
-1 & 1 & 1\end{bmatrix}$.
\end{center}
(iv) Woronowicz algebra $\cW_{\nu}(\mathfrak{sl}(2,K))$, where $\nu \in K-\{0\}$ is not a root of
unity:
\begin{center}
$x_1x_3-\nu^4x_3x_1=(1+\nu^2)x_1,\ \ \ x_1x_2-\nu^2x_2x_1=\nu x_3,\ \ \
x_3x_2-\nu^4x_2x_3=(1+\nu^2)x_2$;

$Gr(B)=K_{\textbf{q}}[x_1,x_2,x_3]$, with $\textbf{q}=\begin{bmatrix}1 & \nu^{-2} & \nu^{-4}\\
\nu^{2} & 1 & \nu^{4}\\
\nu^{4} & \nu^{-4} & 1\end{bmatrix}$.
\end{center}

(v) Eight types of $3$-dimensional skew polynomial algebras, with $\alpha, \beta, \gamma \in
K-\{0\}$:
\begin{center}
$x_2x_3-x_3x_2=x_3$,\ \ $x_3x_1-\beta x_1x_3=x_2$,\ \ $x_1x_2-x_2x_1=x_1$;

$x_2x_3-x_3x_2=0,\ \ x_3x_1-\beta x_1x_3=x_2,\ \ x_1x_2-x_2x_1=0$;

$x_2x_3-x_3x_2=x_3,\ \ x_3x_1-\beta x_1x_3=0,\ \ x_1x_2-x_2x_1=x_1$;

$x_2x_3-x_3x_2=x_3,\ \ x_3x_1-\beta x_1x_3=0,\ \ x_1x_2-x_2x_1=0$;

$x_2x_3-x_3x_2=x_1,\ \ x_3x_1-x_1x_3=x_2,\ \ x_1x_2-x_2x_1 =x_3$;

$x_2x_3-x_3x_2=0,\ \ x_3x_1-x_1x_3=0,\ \ x_1x_2-x_2x_1=x_3$;

$x_2x_3-x_3x_2=-x_2,\ \ x_3x_1-x_1x_3=x_1+x_2,\ \ x_1x_2-x_2x_1=0$;

$x_2x_3-x_3x_2=x_3,\ \ x_3x_1-x_1x_3=x,\ \ x_1x_2-x_2x_1=0 $;

$Gr(B)=K_{\textbf{q}}[x_1,x_2,x_3]$, where $\textbf{q}$ is an appropiate matrix in every case.
\end{center}
Observe that in every example, $Gr(B)$ is a noetherian Artin-Schelter regular algebra, and hence,
$Gr(B)$ satisfies the $\mathcal{X}_1$ condition (see \cite{Rogalski}). From this we will conclude
that $\underline{B}$ also satisfies such condition.

In fact, note first that in general there is an injective $\mathbb{N}$-graded homomorphism of
$B_0$-algebras $\eta:\underline{B}\to Gr(B)$ defined by
\begin{align*}
\bigoplus_{d=0}^\infty Hom_{{\rm sgr}-B}(B,B(d))& \xrightarrow{\eta}\bigoplus_{d=0}^\infty
Gr(B)_d=\bigoplus_{d=0}^\infty
\frac{B_0\oplus\cdots\oplus B_d}{B_0\oplus\cdots\oplus B_{d-1}}\\
f_0+\cdots+f_d & \mapsto \overline{f_0(1)}+\cdots+\overline{f_d(1)},
\end{align*}
with $f_i\in Hom_{{\rm sgr}-B}(B,B(i))$, $0\leq i\leq d$. We only check that $\eta$ is
multiplicative, the other conditions can be proved also easily: $\eta(f_n\star
g_m)=\eta(s^n(g_m)\circ f_n)=\overline{(s^n(g_m)\circ
f_n)(1)}=\overline{s^n(g_m)(f_n(1))}=\overline{g_m(f_n(1))}=
\overline{f_n(1)g_m(1)}=\overline{f_n(1)}\ \overline{ g_m(1)}=\eta(f_n)\eta(g_m)$.

Thus, in the examples above $K=B_0$, $(\underline{B})_0\cong B_0\cong Gr(B)_0$ and the kernel and
cokernel of $\eta$ are right bounded, so we can apply the part (5) of Lemma 8.2 in \cite{Artin2}
and conclude that $\underline{B}$ satisfies $\mathcal{X}_1$.

We finish remarking that for the listed examples we can apply Proposition \ref{proposition6.12} and
Theorem \ref{theorem16.7.13} and obtain that
\begin{center}
${\rm qgr}-K[x_1,x_2,x_3]\simeq{\rm qgr}-\Gamma(\pi(B))_{\geq 0}$, with $B=\cU(\mathcal{G})$;

${\rm qgr}-K_{\textbf{q}}[x_1,x_2,x_3]\simeq{\rm qgr}-\Gamma(\pi(B))_{\geq 0}$,
\end{center}
with $B=\mathcal{U}'(so(3,K)), \cU(osp(1,2)), \cW_{\nu}(\mathfrak{sl}(2,K))$ or any of eight types
of $3$-dimensional skew polynomial algebras above, and $\textbf{q}$ an appropiate matrix in every
case.
\end{example}

\begin{center}
\textbf{Acknowledgements}
\end{center}
The authors are grateful to Efim Zelmanov, Ivan Shestakov, Iryna Kashuba, Daniel Rogalski, Andrea
Solotar, Mariano Suárez, Pablo Zadunaisky, Blas Torrecillas y Juan Cuadra for valuable comments and
suggestions.


\end{document}